\pdfoutput=1
\documentclass[
paper=letter,%
pagesize,%
numbers=noendperiod,%
captions=nooneline,%
abstracton%
]{scrartcl}

\usepackage[T1]{fontenc}%
\usepackage{lmodern}%
\usepackage[american]{babel}%
\usepackage{microtype}%

\usepackage[
hyperref,%
table%
]{xcolor}%

\usepackage{scrlayer-scrpage}%

\usepackage{eso-pic}%
\usepackage{rotating}%
\usepackage{amsmath}%
\usepackage{mathtools}%
\usepackage{amssymb}%
\usepackage{amsthm}%
\usepackage{thmtools}%
\usepackage{etoolbox}%
\usepackage{bm}%
\usepackage{bbm}%
\usepackage{enumitem}%
\usepackage{graphicx}%
\usepackage{grffile}%
\usepackage{tikz}%
\usepackage{wrapfig}%
\usepackage{tabularx}%
\usepackage{siunitx}%
\usepackage{booktabs}%
\usepackage{multirow}%
\usepackage{vruler}%
\usepackage{fancyvrb}%
\usepackage{listings}%
\usepackage{csquotes}%
\usepackage[
backend=bibtex,%
style=authoryear,%
dashed=false,%
uniquename=init,%
firstinits=true,%
hyperref=true,%
useprefix=true,%
maxcitenames=2,%
maxbibnames=6,%
date=iso8601,%
urldate=iso8601%
]{biblatex}%
\usepackage[
hypertexnames=false,%
setpagesize=false,%
pdfborder={0 0 0},%
pdfstartview=Fit,%
bookmarksopen=true,%
bookmarksnumbered=true%
]{hyperref}%
\xdefinecolor{lightgray}{RGB}{247, 247, 247}%
\xdefinecolor{semilightgray}{RGB}{240, 240, 240}%
\xdefinecolor{middlegray}{RGB}{127, 127, 127}%
\xdefinecolor{blue}{RGB}{58, 95, 205}%
\xdefinecolor{deepskyblue}{RGB}{0, 154, 205}%
\xdefinecolor{chocolate}{RGB}{205, 102, 29}%

\setkomafont{pageheadfoot}{\normalfont\normalcolor\sffamily}%
\setkomafont{pagenumber}{\normalfont\normalcolor\sffamily}%
\automark{section}%
\setkomafont{captionlabel}{\normalfont\normalcolor\sffamily\bfseries}%

\setlist{%
  align=left,%
  labelsep=*,%
  leftmargin=*,%
  topsep=1mm,%
  itemsep=0mm%
}
\newcommand*{\mysquare}{\rule[0.18em]{0.36em}{0.36em}}
\newcommand*{\mytriangle}{\raisebox{0.12em}{\resizebox{0.48em}{0.48em}{$\blacktriangleright$}}}
\newcommand*{\mybar}{\rule[0.32em]{0.62em}{0.08em}}
\newcommand*{\mydot}{\raisebox{0.14em}{\resizebox{0.44em}{!}{$\bullet$}}}
\setlist[itemize,1]{label={\mysquare\ }}%
\setlist[itemize,2]{label={\mytriangle\ }}%
\setlist[itemize,3]{label={\mybar\ }}%
\setlist[itemize,4]{label={\mydot\ }}%
\setlist[enumerate,1]{label=\arabic*)}%
\setlist[enumerate,2]{label=\arabic{enumi}.\arabic*)}%
\setlist[enumerate,3]{label=\arabic{enumi}.\arabic{enumii}.\arabic*)}%

\makeatletter
\newcommand\myisodate{\number\year-\ifcase\month\or 01\or 02\or 03\or 04\or 05\or 06\or 07\or 08\or 09\or 10\or 11\or 12\fi-\ifcase\day\or 01\or 02\or 03\or 04\or 05\or 06\or 07\or 08\or 09\or 10\or 11\or 12\or 13\or 14\or 15\or 16\or 17\or 18\or 19\or 20\or 21\or 22\or 23\or 24\or 25\or 26\or 27\or 28\or 29\or 30\or 31\fi}%
\makeatother
\newcommand*{\abstractnoindent}{}%
\let\abstractnoindent\abstract
\renewcommand*{\abstract}{\let\quotation\quote\let\endquotation\endquote
  \abstractnoindent}
\deffootnote[1em]{1em}{1em}{\textsuperscript{\thefootnotemark}}%
\lstdefinestyle{input}{
  backgroundcolor=\color{semilightgray},%
  commentstyle=\itshape\color{chocolate},%
  keywordstyle=\color{blue},%
  stringstyle=\color{blue},%
  numbers=left,%
  numbersep=4.8pt,%
  numberstyle=\color{darkgray!80}\tiny%
}
\lstdefinestyle{output}{
  backgroundcolor=\color{lightgray}%
}

\lstdefinestyle{Lstyle}{
  language=[LaTeX]TeX,%
  texcs={},%
  otherkeywords={}%
}

\lstnewenvironment{Linput}[1][]{%
  \lstset{style=input, style=Lstyle}
  #1%
}{\vspace{-0.25\baselineskip}}%

\lstnewenvironment{Loutput}[1][]{%
  \lstset{style=output, style=Lstyle}
  #1%
}{\vspace{-0.25\baselineskip}}%

\lstdefinestyle{Rstyle}{
  language=R,%
  keywords={if, else, repeat, while, function, for, in, next, break},%
  otherkeywords={}%
}

\lstnewenvironment{Rinput}[1][]{%
  \lstset{style=input, style=Rstyle}
  #1%
}{\vspace{-0.25\baselineskip}}%

\lstnewenvironment{Routput}[1][]{%
  \lstset{style=output, style=Rstyle}
  #1%
}{\vspace{-0.25\baselineskip}}%

\DeclareNameAlias{sortname}{last-first}%
\DefineBibliographyExtras{american}{\DeclareQuotePunctuation{}}%
\renewbibmacro*{volume+number+eid}{%
  \setunit*{\addcomma\space}%
  \printfield{volume}%
  \printfield{number}}
\DeclareFieldFormat*{number}{(#1)}
\DeclareFieldFormat*{title}{#1}%
\DeclareFieldFormat{doi}{%
  \ifhyperref
    {\href{http://dx.doi.org/#1}{\nolinkurl{doi:#1}}}%
    {\nolinkurl{doi:#1}}}%
\renewbibmacro*{in:}{}%
\DeclareFieldFormat{isbn}{ISBN #1}%
\DeclareFieldFormat{pages}{#1}%
\DeclareFieldFormat{url}{\url{#1}}%
\DeclareFieldFormat{urldate}{\mkbibparens{#1}}%
\renewcommand*{\cite}[2][]{\textcite[#1]{#2}}%

\newif\ifstarttheorem
\declaretheoremstyle[%
  spaceabove=0.5em,
  spacebelow=0.5em,
  headfont=\sffamily\bfseries\global\starttheoremtrue,
  notefont=\sffamily\bfseries,
  notebraces={(}{)},
  headpunct={},
  bodyfont=\normalfont,
  postheadspace=\newline%
]{mythmstyle}
\declaretheorem[style=mythmstyle, numberwithin=section]{definition}%
\declaretheorem[style=mythmstyle, sibling=definition]{proposition}
\declaretheorem[style=mythmstyle, sibling=definition]{lemma}
\declaretheorem[style=mythmstyle, sibling=definition]{theorem}
\declaretheorem[style=mythmstyle, sibling=definition]{corollary}
\declaretheorem[style=mythmstyle, sibling=definition]{remark}
\declaretheorem[style=mythmstyle, sibling=definition]{example}
\declaretheorem[style=mythmstyle, sibling=definition]{algorithm}
\declaretheorem[style=mythmstyle, sibling=definition]{assumption}

\makeatletter
\preto\itemize{%
  \if@inlabel
    \ifstarttheorem
      \mbox{}\par\nobreak\vskip\glueexpr-\parskip-\baselineskip+0.25em\relax\hrule\@height\z@
    \fi%
  \fi%
  \global\starttheoremfalse%
 \def\tempa{proof}%
 \ifx\tempa\mycurrenvir
    \ifstarttheorem
      \mbox{}\par\nobreak\vskip\glueexpr-\parskip-\baselineskip+0.25em\relax\hrule\@height\z@
    \fi%
 \fi%
 \global\starttheoremfalse%
}
\preto\enditemize{\global\starttheoremfalse}
\makeatother

\makeatletter
\preto\enumerate{%
  \if@inlabel
    \ifstarttheorem
      \mbox{}\par\nobreak\vskip\glueexpr-\parskip-\baselineskip+0.25em\relax\hrule\@height\z@
    \fi%
  \fi%
  \global\starttheoremfalse%
 \def\tempa{proof}%
 \ifx\tempa\mycurrenvir
    \ifstarttheorem
      \mbox{}\par\nobreak\vskip\glueexpr-\parskip-\baselineskip+0.25em\relax\hrule\@height\z@
    \fi%
 \fi%
 \global\starttheoremfalse%
}
\preto\endenumerate{\global\starttheoremfalse}
\makeatother

\newcommand{\ou}[3]{%
  \mathrel{%
    \vcenter{\offinterlineskip
      \ialign{##\cr$#1$\cr\noalign{\kern-#3}$#2$\cr}%
    }%
  }%
}
\newcommand*{\omu}[3]{\underset{#3}{\overset{#1}{#2}}}
\newcommand*{\omuc}[3]{\underset{\mathclap{#3}}{\overset{\mathclap{#1}}{#2}}}%

\newcommand*{\psiis}[1]{{\psi_{#1}^{-1}}}

\newcommand*{\isim}{\omu{\text{\tiny{ind.}}}{\sim}{}}

\newcommand*{\IN}{\mathbb{N}}

\newcommand*{\IR}{\mathbb{R}}

\newcommand*{\Anc}{\operatorname{Anc}}

\newcommand*{\Exp}{\operatorname{Exp}}

\newcommand*{\N}{\operatorname{N}}

\renewcommand*{\S}{\operatorname{S}}
\newcommand*{\PS}{\operatorname{PS}}

\newcommand*{\I}{\mathbbm{1}}
\newcommand*{\rd}{\mathrm{d}}

\newcommand*{\sign}{\operatorname{sign}}

\newcommand*{\D}{\operatorname{D}}

\newcommand*{\LS}{\mathcal{LS}}
\newcommand*{\LSi}{\LS^{-1}}

\renewcommand*{\P}{\mathbbm{P}}
\newcommand*{\E}{\mathbb{E}}

\newcommand*{\Cov}{\operatorname{Cov}}

\newcommand*{\psii}{{\psi^{-1}}}

\newcommand*{\R}{\textsf{R}}
\newcommand*{\eps}{\varepsilon}

\begin{document}
\thispagestyle{plain}
\begin{center}
  \sffamily
  {\bfseries\LARGE Hierarchical Archimax copulas\par}
  \bigskip\smallskip
  {\Large Marius Hofert\footnote{Department of Statistics and Actuarial Science, University of
    Waterloo, 200 University Avenue West, Waterloo, ON, N2L
    3G1,
    \href{mailto:marius.hofert@uwaterloo.ca}{\nolinkurl{marius.hofert@uwaterloo.ca}}. The
    author would like to thank NSERC for financial support for this work through Discovery
    Grant RGPIN-5010-2015.},
    Rapha\"{e}l Huser\footnote{Computer, Electrical and Mathematical Science and
      Engineering Devision, King Abdullah University of Science and Technology, \href{mailto:raphael.huser@kaust.edu.sa}{\nolinkurl{raphael.huser@kaust.edu.sa}}.},
    Avinash Prasad\footnote{Department of Statistics and Actuarial Science, University of
    Waterloo, 200 University Avenue West, Waterloo, ON, N2L
    3G1,
    \href{mailto:a2prasad@uwaterloo.ca}{\nolinkurl{a2prasad@uwaterloo.ca}}.}
    \par
    \bigskip
    \myisodate\par}
\end{center}
\par\smallskip
\begin{abstract}
  The class of Archimax copulas is generalized to hierarchical Archimax copulas
  in two ways. First, a hierarchical construction of $d$-norm generators is
  introduced to construct hierarchical stable tail dependence functions which
  induce a hierarchical structure on %
  Archimax copulas. Second, by itself or additionally, hierarchical frailties
  are introduced to extend Archimax copulas to hierarchical Archimax copulas in a similar way as
  nested Archimedean copulas extend Archimedean copulas. Possible extensions to
  nested Archimax copulas are discussed. Additionally, a general formula for the
  density and its evaluation of Archimax copulas is introduced.
\end{abstract}
\minisec{Keywords}
Archimedean copulas, nested Archimedean copulas, extreme-value copulas, Archimax
copulas, hierarchies, nesting.
\minisec{MSC2010}
60E05, 62E15, 62H99%

\section{Introduction}
The class of Archimax copulas, see \cite{caperaafougeresgenest2000} and
\cite{charpentierfougeresgenestneslehova2014}, generalizes Archimedean copulas
to incorporate a stable tail dependence function as known from extreme-value
copulas. As special cases, Archimax copulas can be Archimedean or extreme-value
copulas and thus extend both of these classes of copulas. They provide a link between
dependence structures arising in multivariate extremes and Archimedean copulas,
which have intuitive and computationally appealing properties. One feature of
Archimedean copulas is that they can be \emph{nested} in the sense that one can
(under assumptions detailed later) plug Archimedean copulas into each other and
still obtain a proper copula. Such a construction is \emph{hierarchical} in the
sense that certain multivariate margins are exchangeable, yet the copula overall is not;
this additional flexibility to allow for (partial) asymmetry over an
exchangeable model is typically used to model components belonging to different
groups, clusters or business sectors. In this work, we raise the following natural
question (see Sections~\ref{sec:HEVC} and \ref{sec:HAXC}):
\begin{quote}
  How can hierarchical Archimax copulas be constructed?
\end{quote}
Since we work with stochastic representations, sampling is also
covered. Constructing nested Archimax copulas is largely an open problem which
we discuss in Appendix~\ref{sec:NAXC}. Moreover, to fill a gap in the literature,
we present a general formula for the density and its evaluation of Archimax
copulas; see Appendix~\ref{app:AXC:dens}.

In what follows, we assume the reader to be familiar with the basics of
Archimedean copulas (ACs) and extreme-value copulas (EVCs); see, for example,
\cite{mcneilneslehova2009} for the
former (from which we also adopt the notation) and \cite[Chapter~6]{jaworskidurantehaerdlerychlik2010} for the latter.

\section{Hierarchical extreme-value copulas via hierarchical stable tail dependence functions}\label{sec:HEVC}
\subsection{Connection between $d$-norms and stable tail dependence functions}\label{sec:HEVC:dnorm:std}
A copula $C$ is an extreme-value copula if and only if it is \emph{max-stable}, that is, if
\begin{align*}
  C(\bm{u})=C(u_1^{1/m},\dots,u_d^{1/m})^m,\quad m\in\IN,\ \bm{u}\in[0,1]^d;
\end{align*}
see, for example,
\cite[Theorem~6.2.1]{jaworskidurantehaerdlerychlik2010}. An extreme-value copula
$C$ can be characterized in terms of its stable tail dependence function $\ell:[0,\infty)^d\,\to\,[0,\infty)$ via
\begin{align}
  C(\bm{u})= \exp(-\ell(-\log u_1,\dots,-\log u_d)), \quad \bm{u} \in [0,1]^d;\label{eq:gen:form:EVC}
\end{align}
see, for example, %
\cite[Section~8.2]{beirlantgoegebeursegersteugels2004} and
\cite[Chapter~6]{jaworskidurantehaerdlerychlik2010}. A characterization of
stable tail dependence functions
$\ell$ (being homogeneous of order 1, being 1 when evaluated at the
unit vectors in $\IR^d$ and being fully
$d$-max decreasing) is given in \cite{ressel2013} %
and \cite{charpentierfougeresgenestneslehova2014}.

Sampling from EVCs is usually quite challenging and time-consuming for the most
popular models. Examples which are comparably easy to sample are Gumbel and
nested Gumbel copulas, the only Archimedean and nested Archimedean EVCs,
respectively, where a stochastic representation is available; see
\cite[Theorem~4.5.2]{nelsen2006}.
\begin{itemize}
\item The \emph{Gumbel} (or \emph{logistic}) \emph{copula} $C$ with parameter
  $\alpha\in(0,1]$ and stable tail dependence function
  $\ell(\bm{x})=(x_1^{1/\alpha}+\dots+x_d^{1/\alpha})^{\alpha}$,
  $\bm{x}\in[0,\infty)^d$, can be sampled using the algorithm of
  \cite{marshallolkin1988}. It utilizes the stochastic representation
  \begin{align}
    \bm{U}=\Bigl(\psi\Bigl(\frac{E_1}{V}\Bigr),\dots,\psi\Bigl(\frac{E_d}{V}\Bigr)\Bigr)\sim C,\label{eq:stoch:rep:AC}
  \end{align}
  where $\psi(t)=\exp(-t^{\alpha})$ is a Gumbel generator,
  $E_1,\dots,E_d\isim\Exp(1)$, independently of the \emph{frailty}
  $V\sim \PS(\alpha)=\S(\alpha,1,\cos^{1/\alpha}(\alpha\pi/2),\I_{\{\alpha=1\}};1)$; see
  \cite[p.~8]{nolan2017} for the parameterization of this $\alpha$-stable distribution.
\item Nested Gumbel copulas, see \cite{tawn1990}, can also be sampled based on a
  stochastic representation corresponding to the nesting structure; see
  \cite{mcneil2008}. The main idea is to replace the single frailty $V$ by a
  sequence of dependent frailties (all $\alpha$-stable for different $\alpha$),
  nested in a specific way; see Section~\ref{sec:HAXC}.
\end{itemize}

For more complicated EVCs, \cite{schlather2002}, \cite{diekermikosch2015}, and
\cite{dombryengelkeoesting2016} have proposed approximate or exact simulation
schemes based on the following stochastic representation of max-stable
processes; see \cite{dehaan1984}, \cite{penrose1992} and \cite{schlather2002}.

\begin{theorem}[Spectral representation of max-stable processes]
  Let $\{W_i(\bm{s})\}_{i=1}^\infty$ be independent copies of the random process $W(\bm{s})$, $\bm{s}\in\mathcal{S}\subseteq\IR^q$, such that $W(\bm{s})\ge 0$
  and $\E(W(\bm{s}))=1$, $\bm{s}\in \mathcal{S}$. Furthermore, let $\{P_i\}_{i=1}^\infty$ be points of a Poisson
  point process on $[0,\infty)$ with intensity $x^{-2}\,\rd x$. Then
  \begin{align}
    Z(\bm s)=\sup_{i\ge 0}\{P_iW_i(\bm{s})\}\label{eq:stoch:rep:maxstable}
  \end{align}
  is a max-stable random process with unit Fr{\'e}chet margins and
  \begin{align}
    \ell(x_1,\dots,x_d)=\E(\max_{1\le j\le d}\{x_j W(\bm{s}_j)\}),\quad x_1,\dots,x_d>0, \label{stail}
  \end{align}
  is the associated stable tail dependence function of the random vector
  $(Z(\bm{s}_1),\dots,Z(\bm{s}_d))$ for fixed $\bm{s}_1,\dots,\bm{s}_d$. Therefore, if a process $Z(\bm{s})$ can be expressed as
  in \eqref{eq:stoch:rep:maxstable}, the distribution function of the random vector $(Z(\bm{s}_1),\dots,Z(\bm{s}_d))$ is
  $\P(Z(\bm{s}_1)\le x_1,\dots,Z(\bm{s}_d)\le x_d)=\exp(-\ell(1/x_1,\dots,1/x_d))$, that
  is, $(Z(\bm{s}_1),\dots,Z(\bm{s}_d))$ has EVC $C$ with stable tail dependence function $\ell$
  and unit Fr\'echet margins $\exp(-1/x_j)$, $j\in\{1,\dots,d\}$.
\end{theorem}

For completeness, Algorithm~\ref{alg:sampling:EVCsSchlather} below describes the
traditional approach for simulating max-stable processes constructed using
\eqref{eq:stoch:rep:maxstable}. This algorithm goes back to \cite{schlather2002}
and provides approximate simulations by truncating the supremum to a
finite number of processes in \eqref{eq:stoch:rep:maxstable}. When the random
process $W(\bm{s})$ is bounded almost surely, a stopping criterion may be
designed to optimally select the number of Poisson points $N$ to perform
exact simulation. For more general exact sampling schemes, we refer to
\cite{diekermikosch2015} and \cite{dombryengelkeoesting2016}.

\begin{algorithm}[Approximate sampling of max-stable processes based on~\eqref{eq:stoch:rep:maxstable}]\label{alg:sampling:EVCsSchlather}
  \begin{enumerate}
  \item Simulate $N$ Poisson points $\{P_i\}_{i=1}^N$ in decreasing order as
    $P_i=1/\sum_{k=1}^i E_k$, $i\in\{1,\ldots,N\}$, where
    $E_k\isim\Exp(1)$, $k\in\{1,\dots,N\}$.
  \item Simulate $N$ independent copies $\{W_i(\bm{s})\}_{i=1}^N$ of the process
    $W(\bm{s})$ at a finite set of locations
    $\bm{s}\in\{\bm{s}_1,\ldots,\bm{s}_d\}$.
  \item For each location $\bm{s}\in\{\bm{s}_1,\ldots,\bm{s}_d\}$, set
    $Z(\bm s)=\max_{1\leq i \leq N}\{P_iW_i(\bm{s})\}$.
  \end{enumerate}
\end{algorithm}

By choosing the spatial domain $\mathcal S$ in \eqref{eq:stoch:rep:maxstable} to be finite and
replacing $W(\bm{s}_1),\dots,W(\bm{s}_d)$ by non-negative random variables
$W_1,\dots,W_d$ with $\E(W_j)=1$, $j\in\{1,\dots,d\}$, thus replacing the
random process $W(\bm{s})$ by the non-negative random vector
$\bm{W}=(W_1,\dots,W_d)$, this representation also provides a characterization
of, and sampling algorithms for, (finite-dimensional) EVCs;
from here on we will adopt this ``vector case'' for $W$ and accordingly for $Z$.

We now turn to the link between max-stable random vectors $(Z_1,\dots,Z_d)$
and $d$-norms as recently described in \cite{aulbach2015}. A norm
$\lVert\cdot\rVert_d$ on $\IR^d$ is called a \emph{$d$-norm} if there exists a
random vector $\bm{W}=(W_1,\dots,W_d)$ with $W_j\ge 0$ and $\E(W_j)=1$,
$j\in\{1,\dots,d\}$, such that
\begin{align}
  \lVert\bm{x}\rVert_d = \E(\underset{1\leq j\leq d}{\max}\{|x_j|W_j\})=\E(\lVert\bm{xW}\rVert_{\infty}),\quad \bm{x}=(x_1,\dots,x_d) \in \IR^d, \label{dnorm}
\end{align}
where $\lVert\cdot\rVert_{\infty}$ denotes the supremum norm and $\bm{xW}$ is
understood componentwise.  In this case, $\bm{W}$ is called \emph{generator} of
$\lVert\cdot\rVert_d$. One can compare \eqref{stail} and \eqref{dnorm} to
identify the correspondence
\begin{align}
  \ell(\bm{x})=\lVert\bm{x}\rVert_d=\E(\lVert\bm{xW}\rVert_{\infty}),\quad\bm{x}\in[0,\infty)^d,\label{dnorm:correspondence}
\end{align}
between $d$-norms and stable tail dependence functions
on $[0,\infty)^d$. Specifying a generator $\bm{W}$ thus defines a stable tail
dependence function which in turn characterizes an EVC. The link
\eqref{dnorm:correspondence} with
$d$-norms provides us with a useful method for constructing and sampling
EVCs which can also be exploited for constructing hierarchical EVCs (HEVCs).

We now provide a few examples of $d$-norm generators for well known copulas
which can serve as building blocks for HEVCs (and, see Section~\ref{sec:HAXC},
hierarchical Archimax copulas).
\begin{example}\label{ex:dnorm:generators:and:stdfun}
  \begin{enumerate}
  \item If $\bm{W}=(1,\dots,1)$ with probability one, then $\lVert\bm{x}\rVert_d=\underset{1\leq j\leq d}{\max}{|x_j|}$. This characterizes comonotonicity, that is, the upper Fr{\'e}chet--Hoeffding bound with stable tail dependence function $\ell(\bm{x})= \max \{x_1,\dots,x_d\}$.
  \item If $\bm{W}$ is a random permutation of $(d,0,\dots,0)\in\IR^d$ , then
    $\lVert\bm{x}\rVert_d = d\sum_{j=1}^{d} |x_j|/d = \sum_{j=1}^{d} |x_j|$. This characterizes independence with
    the stable dependence function $\ell(\bm{x})=x_1+ \dots + x_d$.
  \item If $\bm{W}=(W_1,\dots,W_d)$ is such that for some $0<\alpha<1$,
    $\Gamma(1-\alpha)W_j\isim \exp(-x^{-1/\alpha})$, $x\in [0,\infty)$, where $\Gamma$
    denotes the gamma function, a straightforward computation shows that
    $\lVert\bm{x}\rVert_d= (\sum_{j=1}^{d} |x_j|^{1/\alpha})^{\alpha}$. This implies
    that $\ell(\bm{x})=(\sum_{j=1}^{d} x_j^{1/\alpha})^{\alpha}$ and thus that
    the max-stable dependence structure is the Gumbel (logistic) copula with
    parameter $\alpha \in (0,1)$.
  \item If $\bm{W}$ is such that for some $\theta>0$,
    $W_j=\Gamma(1+1/\theta)W_j^{*}$ with $W_j^{*}\isim\exp(-x^{\theta})$,
    $x\in [0,\infty)$, then the stable tail dependence function can be
    calculated to be
    \begin{align*}
      \ell(\bm{x})=\sum_{\emptyset\neq J\subseteq \{1,\dots,d\} }(-1)^{|J|+1} \biggl(\,\sum_{j\in J} x_j^{-\theta}\biggr)^{-1/\theta},
    \end{align*}
    and thus the max-stable dependence structure is the negative logistic copula
    with parameter $\theta>0$; see, for example,
    \cite{dombryengelkeoesting2016}.
  \item If
    $\bm{W}=(W_1,\dots,W_d)\sim
    (\sqrt{2\pi}\max\{0,\eps_1\},\dots,\sqrt{2\pi}\max\{0,\eps_d\})$,
    where $(\eps_1,\dots,\eps_d) \sim \N_d(\bm{0},P)$ with correlation matrix
    $P$, a \emph{Schlather model} results; see \cite{schlather2002}.
  \item\label{ex:dnorm:generators:and:stdfun:6} If
    $\bm{W}=(W_1,\dots,W_d)\sim
    (\max\{0,\eps_1\}^{\nu}/c_{\nu},\dots,\max\{0,\eps_d\}^{\nu}/c_{\nu})$,
    where $(\eps_1,\dots,\eps_d) \sim \N_d(\bm{0},P)$ with correlation matrix
    $P$, $\nu>0$, and $c_{\nu}=2^{\nu/2-1}\Gamma((\nu+1)/2)/\sqrt{\pi}$,
    then the
    \emph{extremal $t$ model} of \cite{opitz2013} results; for $\nu=1$, the
    Schlather model is obtained as a special case. The stable tail dependence function
    $\ell(\bm{x})$ of the extremal $t$ model in dimension $d$ is given by
    \begin{align}
      \ell(\bm{x})=\sum_{j=1}^d x_j t_{d-1}\Bigl(\nu+1,\ P_{-j,j},\
      \frac{P_{-j,-j}-P_{-j,j}P_{j,-j}}{\nu+1}\Bigr)((\bm{x}_{-j}/x_j)^{-1/\nu}),\label{eq:et:std}
    \end{align}
    where $t_{d}(\nu,\bm{\mu},\Sigma)(\bm{x})$ denotes the $d$-variate Student $t$
    distribution function with $\nu$ degrees of freedom, location vector
    $\bm{\mu}$ and dispersion matrix $\Sigma$ evaluated at $\bm{x}$ as in
    \cite[Example~6.7]{mcneilfreyembrechts2015}, $P_{-j,-j}$ (respectively,
    $P_{-j,j}$, $P_{j,-j}$) denotes the
    submatrix obtained by removing the $j$th row and the $j$th column
    (respectively, $j$th
    row, $j$th column) from $P$ and $\bm{x}_{-j}=(x_1,\dots,x_{j-1},x_{j+1},\dots,x_d)$.
  \item\label{ex:dnorm:generators:and:stdfun:7} If
    $\bm{W}=(W_1,\dots,W_d)\sim
    (\exp(\eps_1-\sigma_1^2/2),\dots,\exp(\eps_d-\sigma_d^2/2)),$
    where $(\eps_1,\dots,\eps_d) \sim \N_d(\bm{0},\Sigma)$ for a covariance
    matrix $\Sigma$ with diagonal entries $\Sigma_{jj}=\sigma^2_j$,
    $j\in\{1,\dots,d\}$, and corresponding correlation matrix $P$ (such that
    $\Sigma_{ij}=\sigma_i\sigma_jP_{ij}$, $i,j\in\{1,\dots,d\}$), a
    \emph{Brown--Resnick model} results; see \cite{kabluchko2009}. This model
    can also be obtained as a certain limit of the extremal $t$ model when the
    degrees of freedom $\nu\to\infty$; see \cite{nikoloulopoulosjoeli2009}. The
    Brown--Resnick model is characterized by the \emph{H\"usler--Reiss copula};
    see \cite{huslerreiss1989}. Its stable dependence function $\ell(\bm{x})$ is
    available in any dimension $d$, see \cite{nikoloulopoulosjoeli2009} and
    \cite{huser2013}, and given by
    \begin{align}
      \ell(\bm{x})=\sum_{j=1}^d x_j \Phi_{d-1}(\bm{0},\Sigma_j)(\bm{\eta}_j),\label{eq:HR:std}
    \end{align}
    where $\Phi_{d}(\bm{\mu}, \Sigma)(\bm{x})$ denotes the $d$-variate
    normal distribution function with mean vector $\bm{\mu}$ and covariance matrix
    $\Sigma$ evaluated at $\bm{x}$,
    $\Sigma_j$ is the $(d-1)\times (d-1)$ covariance matrix with entries
    \begin{align*}
      \Sigma_{j,ik}=\begin{cases}
        2\gamma_{ij},&\text{if}\,\ k=i\in\{1,\dots,d\}\backslash\{j\},\\
        \gamma_{ij}+\gamma_{jk}-\gamma_{ik},&\text{if}\,\ k\neq i,
      \end{cases}
    \end{align*}
    where $\gamma_{ij}=\sigma_i^2+\sigma_j^2-\sigma_i\sigma_j P_{ij}$, and
    $\bm{\eta}_j$ is the $(d-1)$-dimensional vector with $i$th entry
    $\gamma_{ij} - \log(x_i/x_j)$.
  \item If $\bm{W}=(W_1,\dots,W_d)\sim H$ for a distribution function $H$ with
    margins $F_1,\dots,F_d$ on $[0,\infty)$ such that $\E(W_j)=1$,
    $j\in\{1,\dots,d\}$, then, by Sklar's Theorem, if $C$ denotes the copula of
    $H$, one can derive the general form of $\ell$ via
    \eqref{dnorm:correspondence}. If $\bm{U}\sim C$, then the stochastic
    representation $\bm{W}=(F_1^-(U_1),\dots,F_d^-(U_d))$ can be used to see that, for all $\bm{x}>\bm{0}$,
    \begin{align*}
      G_{\bm{x}}(y)&=\P(\underset{1\leq j\leq d}{\max}\{|x_j|W_j\}\le y)=\P(W_1\le y/x_1,\dots,W_d\le y/x_d)\\
                 &=\P(U_1\le F_1(y/x_1),\dots,U_d\le F_d(y/x_d))=C(F_1(y/x_1),\dots,F_d(y/x_d)).
    \end{align*}
    Applying the chain rule for differentiating this expression with respect to $y$ leads to the density
    \begin{align*}
      g_{\bm{x}}(y) = \sum_{j=1}^d \D_j C(F_1(y/x_1),\dots,F_d(y/x_d))f_j(y/x_j)/x_j,
    \end{align*}
    where $\D_j C(\bm{u})$ denotes the partial derivatives of $C$ with
    respect to the $j$th argument evaluated at $\bm{u}$. By \eqref{dnorm:correspondence} and the substitution $z_j=y/x_j$, we thus have
    that, for all $\bm{x}>\bm{0}$,
    \begin{align*}
      \ell(\bm{x})&=\int_0^\infty y g_{\bm{x}}(y)\,\rd y=\sum_{j=1}^d\frac{1}{x_j}\int_0^\infty y \D_j C(F_1(y/x_1),\dots,F_d(y/x_d))f_j(y/x_j)\,\rd y\\
                  &=\sum_{j=1}^dx_j\int_0^\infty z_j \D_j C(F_1(z_j x_j/x_1),\dots,F_d(z_j x_j/x_d))f_j(z_j)\,\rd z_j\\
      &=\sum_{j=1}^dx_j\E\bigl(Z_j \D_j C(F_1(Z_j x_j/x_1),\dots,F_d(Z_j x_j/x_d))\bigr),
    \end{align*}
    where $Z_1\sim F_1,\dots, Z_d\sim F_d$ are independent. This formula
    resembles \eqref{eq:et:std} and \eqref{eq:HR:std}. If required, it can be
    evaluated by Monte Carlo, for example. Note that it only poses a restriction
    on the marginal distributions (being non-negative and scalable to have mean 1), not the
    dependence of the components of $\bm{W}$.
  \end{enumerate}
\end{example}

\subsection{Hierarchical stable tail dependence functions}
Let us now turn to a construction method for HEVCs by exploiting the link
between $d$-norm generators and stable tail dependence functions established in
Section~\ref{sec:HEVC:dnorm:std}. The idea is to build stable tail dependence functions
with a hierarchical structure on the level of the associated $d$-norm
generator. Although our approach is similar in spirit to \cite{LeeJoe2017} who
recently proposed factor extreme-value copula models, the two constructions
differ.

By analogy with the construction of nested Archimedean copulas (outlined in
Section~\ref{sec:HAXC}) we define hierarchical $d$-norm generators
$\bm{W}=(W_1,\dots,W_d)$ in terms of a tree structure with $d$ leaves. Under
this framework, each component $W_j$, $j\in\{1,\dots,d\}$, is obtained as a
measurable, non-negative function $g_j$ of intermediate variables
$\{W^{*}_{k}\}_{k\in\Anc(j)}$, lying on the tree nodes along the path from the
seed $W_0^{*}$ at the root of the tree to the $j$th leaf represented by the
variable $W_j$ itself. In other words, the variable $W_j$ may be expressed in
terms of its ancestor variables identified by the index set $\Anc(j)$, some
of which may be shared with other variables $W_k$, $k\neq j$, thus inducing
dependence between the components of the vector $\bm{W}$. To fix ideas, consider the
tree represented in Figure~\ref{fig:NEVC:tree}.
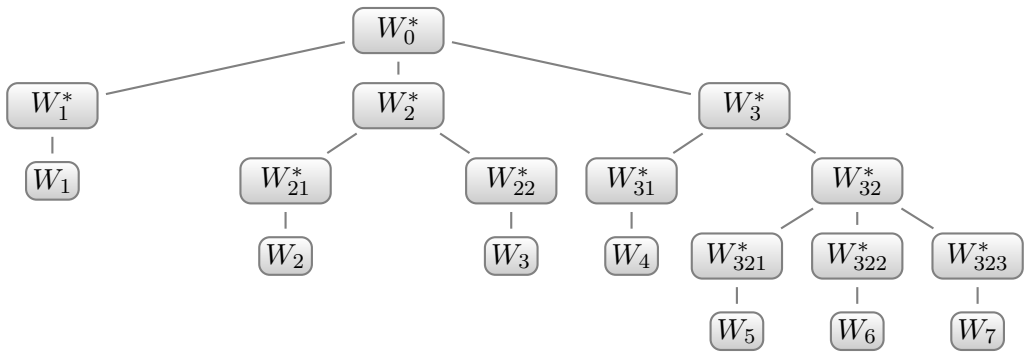
\begin{figure}[htbp]
  \centering
  \begin{tikzpicture}[
    grow = south,
    level 1/.style = {sibling distance = 46mm, level distance = 10mm},
    level 2/.style = {sibling distance = 30mm, level distance = 10mm},
    level 3/.style = {sibling distance = 16mm, level distance = 10mm},
    edge from parent/.style = {thick, draw = black!50, shorten >= 1mm, shorten <= 1mm},
    every node/.style = {rectangle, rounded corners, shade, top color = white,
      bottom color = black!20, draw = black!50, thick, inner sep = 0mm},
    mynodestyle/.style = {minimum width = 12mm, minimum height = 6mm},
    myleafstyle/.style = {minimum width = 7mm, minimum height = 5mm}
    ]
    \node[mynodestyle]{$W_0^{*}$}
    child[edge from parent/.style = {thick, draw = black!50, shorten >= 1mm, shorten <= 1mm}]{
      node[mynodestyle]{$W_{1}^{*}$}
      child[edge from parent/.style = {thick, draw = black!50, shorten >= 1mm, shorten <= 1mm}]{
               node[myleafstyle]{$W_{1}$}
             }
    }
    child[edge from parent/.style = {thick, draw = black!50, shorten >= 1mm, shorten <= 1mm}]{
      node[mynodestyle]{$W_{2}^{*}$}{
        child[edge from parent/.style = {thick, draw = black!50, shorten >= 1mm, shorten <= 1mm}]{
          node[mynodestyle]{$W_{21}^{*}$}
          child[edge from parent/.style = {thick, draw = black!50, shorten >= 1mm, shorten <= 1mm}]{
               node[myleafstyle]{$W_{2}$}
             }
        }
        child[edge from parent/.style = {thick, draw = black!50, shorten >= 1mm, shorten <= 1mm}]{
          node[mynodestyle]{$W_{22}^{*}$}
          child[edge from parent/.style = {thick, draw = black!50, shorten >= 1mm, shorten <= 1mm}]{
               node[myleafstyle]{$W_{3}$}
             }
        }
      }
    }
    child[edge from parent/.style = {thick, draw = black!50, shorten >= 1mm, shorten <= 1mm}]{
      node[mynodestyle]{$W_{3}^{*}$}{
        child[edge from parent/.style = {thick, draw = black!50, shorten >= 1mm, shorten <= 1mm}]{
          node[mynodestyle]{$W_{31}^{*}$}
          child[edge from parent/.style = {thick, draw = black!50, shorten >= 1mm, shorten <= 1mm}]{
               node[myleafstyle]{$W_{4}$}
             }
        }
        child[edge from parent/.style = {thick, draw = black!50, shorten >= 1mm, shorten <= 1mm}]{
          node[mynodestyle]{$W_{32}^{*}$}{
            child[edge from parent/.style = {thick, draw = black!50, shorten >= 1mm, shorten <= 1mm}]{
              node[mynodestyle]{$W_{321}^{*}$}
              child[edge from parent/.style = {thick, draw = black!50, shorten >= 1mm, shorten <= 1mm}]{
               node[myleafstyle]{$W_{5}$}
             }
            }
            child[edge from parent/.style = {thick, draw = black!50, shorten >= 1mm, shorten <= 1mm}]{
              node[mynodestyle]{$W_{322}^{*}$}
              child[edge from parent/.style = {thick, draw = black!50, shorten >= 1mm, shorten <= 1mm}]{
               node[myleafstyle]{$W_{6}$}
             }
            }
            child[edge from parent/.style = {thick, draw = black!50, shorten >= 1mm, shorten <= 1mm}]{
              node[mynodestyle]{$W_{323}^{*}$}
              child[edge from parent/.style = {thick, draw = black!50, shorten >= 1mm, shorten <= 1mm}]{
               node[myleafstyle]{$W_{7}$}
             }
            }
          }
        }
      }
    };
  \end{tikzpicture}
  \caption{Tree representation of a hierarchical $d$-norm generator with $d=7$ for the construction of a HEVC.}
  \label{fig:NEVC:tree}
\end{figure}
In this case, one has, for example, $W_2=g_2(W_0^{*},W_{2}^{*},W_{21}^{*})$ and $W_7=g_7(W_0^{*},W_{3}^{*},W_{32}^{*},W_{323}^{*})$.
To define a valid $d$-norm generator, we need to assume that this system of
variables and the corresponding functions $g_j$ are such that $\E(W_j)=1$ for each
$j\in\{1,\dots,d\}$. However, there is no further restriction on the
dependence structure of these latent variables, which yields a very general
framework.

The inherent hierarchical structure of the $d$-norm generator defined in this
way carries over to the EVC derived from \eqref{stail}. Such hierarchical
$d$-norm generators yield HEVCs.

We now describe several example models of HEVCs constructed using this general
framework. We first consider the well known nested Gumbel copula and show that it arises as HEVCs in our framework; see \cite{mcfadden1978}, \cite{tawn1990} and \cite{stephenson2003} for early references. Nested Gumbel (or logistic) copulas have
been applied in a variety of applications, such as \cite{hofertscherer2011} in
the realm of pricing collateralized debt obligations or
\cite{vettorihusergenton2017} where they are used to group various air
pollutants into clusters with homogeneous extremal dependence strength.

\clearpage
\begin{example}[Nested Gumbel copulas with two nesting levels]\label{ex:2:level:NAGC}
  For $0<\alpha_1,\dots,\alpha_S\leq\alpha_0\le 1$, consider independent random
  variables organized in $S$ groups:
  \begin{alignat*}{2}
    \text{Root:}& \quad & W_0^{*}&=1,\\
    \text{Level~1:}& \quad & W_s^{*}&\isim\PS(\alpha_s/\alpha_0),\quad s\in\{1,\dots,S\},\\
    \text{Level~2:}& \quad & W_{sj}^{*}&\isim\exp(-x^{-1/\alpha_{s}}),\
    x>0,\quad s\in \{1,\dots,S\},\ j\in\{1,\dots,d_s\}.
  \end{alignat*}
  As outlined above, the leaves of the tree correspond to the $d$-norm generator
  $\bm{W}=(\bm{W}_1,\dots,\bm{W}_S)$, with $\bm{W}_s=(W_{s1},\dots,W_{sd_s})$,
  $s\in\{1,\dots,S\}$, with $d=\sum_{s=1}^Sd_s$, where
  \begin{align}
    W_{sj} = g_{sj}(W_0^{*},W_s^{*},W_{sj}^{*})=\frac{W_s^{*\alpha_s}W_{sj}^{*}}{\Gamma(1-\alpha_0)},\quad s \in
    \{1,\dots,S\},\ j\in \{1,\dots,d_s\}.\label{NGC:dnorm:generator}
  \end{align}
  It can be verified that, indeed, $W_{sj}\ge 0$ and $\E(W_{sj})=1$ for all $s$
  and $j$. Then, the stable tail dependence function corresponding to the
  $d$-norm generator $\bm{W}$ is given by
  \begin{align}
    \ell(\bm{x})=\ell_{\alpha_0}(\ell_{\alpha_1}(\bm{x}_1),\dots,\ell_{\alpha_S}(\bm{x}_S)),\label{eq:stabletaildep:NestedGumbel}
  \end{align}
  where $\bm{x}=(\bm{x}_1,\dots,\bm{x}_S)$, $\bm{x}_s=(x_{s1},\dots,x_{sd_s})$,
  $s\in\{1,\dots,S\}$, and
  $\ell_\alpha(x_1,\dots,x_d)=(\sum_{j=1}^d x_j^{1/\alpha})^\alpha$ is the
  stable tail dependence function of a Gumbel copula with parameter $\alpha$.
  \begin{proof}
    It directly follows from \eqref{dnorm:correspondence} that
    \begin{align*}
      \ell(\bm{x})=\E(\max_{1 \le s \le S}\{\max_{1 \le j \le d_s}\{x_{sj}W_{sj} \}\}),\quad \bm{x}\in[0,\infty)^d.
    \end{align*}
    By \eqref{NGC:dnorm:generator} and with $Y_{\bm{x}}=\max_{1 \le s \le S}\{\max_{1 \le j \le
      d_s}\{x_{sj}W_{s}^{*\alpha_s}W_{sj}^{*} \}\}$, one obtains that
    \begin{align*}
      \ell(\bm{x})&=\frac{1}{\Gamma(1-\alpha_0)} \E\bigl(\max_{1 \le s \le
                    S}\{\max_{1 \le j \le d_s}\{x_{sj}W_{s}^{*\alpha_s}W_{sj}^{*}
                    \}\}\bigr)=\frac{1}{\Gamma(1-\alpha_0)}\E(Y_{\bm{x}})\\
                  &=\frac{1}{\Gamma(1-\alpha_0)} \int_0^{\infty}
                    \P(Y_{\bm{x}}>y)\,\rd y,\quad \bm{x}\in[0,\infty)^d.
    \end{align*}
    Conditioning on $W_s^{*}$, $s \in \{1,\dots,S\}$, we obtain that
    \begin{align*}
      \P(Y_{\bm{x}}\leq y)& =\P(W_{sj}^{*} \leq y/(x_{sj}W_s^{*\alpha_s}),\
                            s\in\{1,\dots,S\},\ j\in\{1,\dots,d_s\})\\
                          &=\E\biggr(\prod_{s=1}^{S}\prod_{j=1}^{d_s} \exp \bigr(\bigr(-\frac {y}{x_{sj}W_s^{*\alpha_s}}\bigr)^{\frac{1}{\alpha_s}}\bigr)\biggr)=\prod_{s=1}^{S} \E\biggr(\,\exp(-W_s^{*}\sum_{j=1}^{d_s}\bigg(\frac{y}{x_{sj}}\bigg)^{-\frac{1}{\alpha_s}}\biggr),
    \end{align*}
    where the last equality holding since $W^{*}_{1},\dots,W^{*}_{S}$ are
    independent. Since $W_s^{*}\sim \PS(\alpha_s/\alpha_0)$, this leads to
    \begin{align*}
      \P(Y_{\bm{x}}\leq y)=\prod_{s=1}^{S} \exp\biggr(-\bigg(\sum_{j=1}^{d_s} \bigg(\frac{y}{x_{sj}}\bigg)^{-\frac{1}{\alpha_s}}\bigg)^\frac{\alpha_s}{\alpha_0}\biggr)
      =\exp\biggr(-y^{-\frac{1}{\alpha_0}}\biggr(\sum_{s=1}^{S}\bigg(\sum_{j=1}^{d_s} ({x_{sj}^{\frac{1}{\alpha_s}}})\bigg)^{\frac{\alpha_s}{\alpha_0}}\biggr)\biggr).
    \end{align*}
    With $t=\sum_{s=1}^{S}(\sum_{j=1}^{d_s}
    ({x_{sj}^{\frac{1}{\alpha_s}}}))^{\frac{\alpha_s}{\alpha_0}}$, the
    substitution $z=y^{-\frac{1}{\alpha_0}}t$, and integration by parts, the stable tail dependence function
    is thus
    \begin{align*}
      \ell(\bm{x})&=\frac{1}{\Gamma(1-\alpha_0)} \int_0^{\infty}
                    (1-\exp(-y^{-\frac{1}{\alpha_0}}t))\,\rd y=\frac{t^{\alpha_0}}{\Gamma(1-\alpha_0)}
                    \int_0^{\infty} (1-\exp(-z))\alpha_0 z^{-\alpha_0-1}\,\rd z\\
                  &=\frac{t^{\alpha_0}}{\Gamma(1-\alpha_0)} \int_0^{\infty} z^{-\alpha_0} \exp(-z)\,\rd z
                    = \frac{t^{\alpha_0}}{\Gamma(1-\alpha_0)} \Gamma(1-\alpha_0)=
                    t^{\alpha_0}\\
                  &=\biggr(\,\sum_{s=1}^{S}\biggl(\,\sum_{j=1}^{d_s}
                    {x_{sj}^{\frac{1}{\alpha_s}}}\biggr)^{\frac{\alpha_s}{\alpha_0}}\biggr)^{\alpha_0}=\ell_{\alpha_0}(\ell_{\alpha_1}(\bm{x}_1),\dots,\ell_{\alpha_S}(\bm{x}_S)),
    \end{align*}
    which is the stable tail dependence function of a nested Gumbel copula constructed by nesting on the level of the $d$-norms.
  \end{proof}
\end{example}

The construction underlying Example~\ref{ex:2:level:NAGC} may easily be
generalized to trees with arbitrary nesting levels using the
same line of proof. The construction, extending \cite{stephenson2003}, is
outlined in Example~\ref{example:NGumbel2}.

\begin{example}[Nested Gumbel copulas with arbitrary nesting
  levels]\label{example:NGumbel2}
  To construct a nested Gumbel copula with arbitrary nesting levels, we
  mimic the construction with two nesting levels in
  Example~\ref{ex:2:level:NAGC}. Let $p_j$ be the path starting from the root
  of the tree and leading to the $j$th leaf representing the $d$-norm generator
  component $W_j$. We can write the corresponding node variables along this path as
  $W_0^{*}, W_{p_j(1)}^{*}, W_{p_j(2)}^{*}, \dots, W_{p_j(L_j)}^{*}, W_j$,
  where $L_j$ denotes the number of intermediate variables (or levels) between
  $W^{*}_0$ and $W_j$. Assume that all latent variables $W_{p_j(k)}^{*}$,
  $j\in\{1,\dots,d\}$, $k\in\{1,\dots,L_j\}$, are mutually independent within
  and across paths, and that
  \begin{alignat*}{2}
    \text{Root:} &\quad & W_0^{*}&=1,\\
    \text{Level $1$:} &\quad & W_{p_j(1)}^{*}&\sim \PS(\alpha_{p_j(1)}/\alpha_0),\\
    \text{Level $k$:} &\quad & W_{p_j(k)}^{*}&\sim \PS(\alpha_{p_j(k)}/\alpha_{p_j(k-1)}),\quad k\in\{2,\dots,L_j-1\},\\
    \text{Level $L_j$:} &\quad& W_{p_j(L_j)}^{*}&\sim \exp(-x^{-1/\alpha_{p_j(L_j-1)}}),\quad x>0,
  \end{alignat*}
  where, for each path $p_j$, the parameters of the
  positive $\alpha$-stable variables on this path are ordered as
  $0<\alpha_{p_j(L_j-1)}\le\dots\le\alpha_{p_j(1)}\le\alpha_0<1$.
  We can then construct the component $W_j$ of the $d$-norm generator via
  \begin{align*}
    W_j=g_{j}(W_0^{*},W_{p_j(1)}^{*},\dots,W_{p_j(L_j)}^{*})=\frac{W_{p_j(1)}^{{*}\alpha_{p_j(1)}}\cdots
    W_{p_j(L_j-1)}^{{*}\alpha_{p_j(L_j-1)}}W_{p_j(L_j)}^{*}}{\Gamma(1-\alpha_0)},\quad j\in\{1,\dots,d\}.
  \end{align*}
  By recursively conditioning on the variables along each path, one can show
  that the resulting $d$-norm generator corresponds to the nested Gumbel copula
  based on the same tree structure and that its stable tail dependence function
  can be obtained by applying \eqref{eq:stabletaildep:NestedGumbel} recursively
  at each nesting level of the tree.
\end{example}

The construction principle for hierarchical $d$-norm generators also allows us to
construct the following two HEVCs.

\begin{example}[Hierarchical H\"usler--Reiss copula]\label{example:HierarchicalHuslerReiss}
  For simplicity, consider the two-level case
  \begin{alignat*}{2}
  \text{Root:} &\quad & W_0^{*}&=1,\\
  \text{Level~1:} &\quad & (W_1^{*},\dots,W_S^{*})&\sim \N_S(\bm{0},\Sigma_0),\\
  \text{Level~2:} &\quad & (W_{s1}^{*},\dots,W_{sd_s}^{*})&\sim \N_{d_s}(\bm{0},\Sigma_s),\quad s\in \{1,\dots,S\},
  \end{alignat*}
  where the vectors $(W_1^{*},\dots,W_S^{*})$ and
  $(W_{s1}^{*},\dots,W_{sd_s}^{*})$, $s\in \{1,\dots,S\}$, are
  independent. Furthermore, assume that the covariance matrix
  $\Sigma_0$ may be expressed in terms of the variances
  $\sigma^{{*}2}_1,\dots,\sigma^{{*}2}_S$ and the correlation matrix $P_0$
  via $\Sigma_{0,ik}=\Cov(W_i^{*},W_k^{*})=\sigma^{{*}}_i\sigma^{{*}}_kP_{0,ik}$. Similarly,
  denote by $\sigma^{{*}2}_{s1},\dots,\sigma^{{*}2}_{sd_s}$ and $P_s$ the
  respective quantities for the vector $(W_{s1}^{*},\dots,W_{sd_s}^{*})$,
  $s\in \{1,\dots,S\}$. Writing the $d$-norm generator as
  $\bm{W}=(\bm{W}_1,\dots,\bm{W}_S)$, with $\bm{W}_s=(W_{s1},\dots,W_{sd_s})$, $s\in\{1,\dots,S\}$,
  as in Example~\ref{ex:2:level:NAGC}, we define the components by
  \begin{align}
    W_{sj}=\exp((W_s^{*}+W_{sj}^{*})-(\sigma_s^{{*}2}+\sigma_{sj}^{{*}2})/2),\quad s\in \{1,\dots,S\},\ j\in\{1,\dots,d_s\}.\label{eq:HierarchicalHuslerReiss}
  \end{align}
  It is immediate from Part~\ref{ex:dnorm:generators:and:stdfun:7} of
  Example~\ref{ex:dnorm:generators:and:stdfun} and by writing
  $\eps_{sj}=W_s^{*}+W_{sj}^{*}$ that the resulting extreme-value distribution
  has the H\"usler--Reiss copula as underlying dependence structure. It is
  characterized by an overall dispersion matrix $\Sigma$ whose entries are given by
  \begin{align*}
   \Cov(\eps_{s_1j_1},\eps_{s_2j_2})=\begin{cases}
   \Sigma_{0,s_1s_1}+\Sigma_{s_1,j_1j_2}=\sigma_{s_1}^{{*}2}+\sigma_{s_1j_1}^{{*}}\sigma_{s_1j_2}^{{*}}P_{s_1,j_1j_2},&s_1=s_2\
   \text{(same groups)},\\
   \Sigma_{0,s_1s_2}=\sigma_{s_1}^{{*}}\sigma_{s_2}^{{*}}P_{0,s_1s_2},&s_1\neq
   s_2\ \text{(different groups)}.\\
  \end{cases}
  \end{align*}
  Hence, in this case, the underlying hierarchical $d$-norm generator results in
  a hierarchical structure of the covariance matrix $\Sigma$ and the
  corresponding stable tail dependence function is of the same form. It is
  straightforward to verify that this hierarchical structure allows to model
  stronger dependence within groups than between groups. This simple two-level
  example can easily be generalized to trees with arbitrary nesting
  levels, and it could be interesting for spatial modeling, where different
  homogeneous regions exhibit different extreme-value behaviors.
\end{example}

\begin{example}[Hierarchical extremal $t$ and Schlather copula]
  Example~\ref{example:HierarchicalHuslerReiss} can be adapted to a hierarchical
  extremal $t$ model by replacing \eqref{eq:HierarchicalHuslerReiss} by
  \begin{align*}
    W_{sj}=\max\{0,(W_s^{*}+W_{sj}^{*})/(\sigma_s^{{*}2}+\sigma_{sj}^{{*}2})^{1/2}\}^\nu/c_\nu,\quad s\in \{1,\dots,S\},\ j\in\{1,\dots,d_s\},
  \end{align*}
  where $\nu>0$ is the degrees of freedom and $c_\nu$ is the same constant
  appearing in Part~\ref{ex:dnorm:generators:and:stdfun:6} of
  Example~\ref{ex:dnorm:generators:and:stdfun}. For $\nu=1$, we obtain a
  hierarchical Schlather model.
\end{example}

\section{Hierarchical Archimax copulas}\label{sec:HAXC}
\subsection{Archimax copulas}
According to \cite{caperaafougeresgenest2000} and
\cite{charpentierfougeresgenestneslehova2014}, a copula is an \emph{Archimax
  copula (AXC)} if it admits the form
\begin{align}
  C(\bm{u})=\psi(\ell(\psii(u_1),\dots,\psii(u_d))),\quad\bm{u}\in[0,1]^d,\label{eq:AXC}
\end{align}
for an Archimedean generator $\psi\in\Psi$ and a stable tail dependence function
$\ell$; note that the form \eqref{eq:AXC} in $d$ dimensions was originally conjectured
in \cite{mesiarjagr2013}. In what follows, we focus on the case where $\psi$ is
completely monotone. Since $\psi(0)=1$, Bernstein's Theorem, see
\cite{bernstein1928} or \cite[p.~439]{feller1971}, implies that $\psi$ is the
Laplace--Stieltjes transform of a distribution function $F$ on the positive real
line, that is, $\psi(t)=\LS[F](t)=\int_0^\infty\exp(-tx)\,\rd F(x)$,
$t\in[0,\infty)$, in this case. A stochastic representation for $\bm{U}\sim C$
is given by
\begin{align}
  \bm{U}=\biggl(\psi\biggl(\frac{E_1}{V}\biggr),\dots,\psi\biggl(\frac{E_d}{V}\biggr)\biggr)=\biggl(\psi\biggl(\frac{-\log
  Y_1}{V}\biggr),\dots,\psi\biggl(\frac{-\log Y_d}{V}\biggr)\biggr)\sim C,\label{eq:stoch:rep:AXC}
\end{align}
where $(E_1,\dots,E_d)=(-\log Y_1,\dots,-\log Y_d)$ (which has $\Exp(1)$
margins) for $\bm{Y}=(Y_1,\dots,$ $Y_d)\sim D$ for a $d$-dimensional EVC $D$ with
stable tail dependence function $\ell$ and $V\sim F=\LSi[\psi]$ is the frailty
in the construction (which is independent of
$\bm{Y}$). %
Note that, as a special case, if $D$ is the independence copula, in
other words $\ell(\bm{x})=\sum_{j=1}^dx_j$, then $C$ in \eqref{eq:AXC} is
Archimedean. Moreover, if $\psi(t)=\exp(-t)$, $t\geq0$, then $C$ in
\eqref{eq:AXC} is an EVC with stable tail dependence function $\ell$ (compare
with \eqref{eq:gen:form:EVC}) and $\bm{U}=\bm{Y}$, so $C=D$. Although not
relevant for the remainder of this paper, but important for statistical
applications, let us mention that, if it exists, the density of an AXC allows
for a rather explicit form (derived in Proposition~\ref{prop:AXC:dens}) which
makes computing the logarithmic density numerically feasible (see
Proposition~\ref{prop:AXC:log:dens:eval}).

\subsection{Two ways of inducing hierarchies}\label{sec:two:ways:HAXCs}
There are two immediate ways to introduce a hierarchical structure on Archimax
copulas following from \eqref{eq:stoch:rep:AXC}, thus leading to
\emph{hierarchical Archimax copulas (HAXCs)}: At the level of the EVC $D$
through its stable tail dependence function (via $d$-norms) and at the level of
the frailty $V$ by using a sequence of dependent frailties instead of a single
$V$. Since the former was addressed in Section~\ref{sec:HEVC}, we now focus on
the latter.

Let $D$ be a $d$-dimensional EVC with stable tail dependence function
$\ell$ as before. The stochastic representation \eqref{eq:stoch:rep:AXC} can be
generalized by replacing the single frailty $V$ by a sequence of dependent
frailties. Their hierarchical structure and dependence is best described in
terms of a concrete example. To this end, consider Figure~\ref{fig:NAC:tree}.
\begin{figure}[htbp]
  \centering
  \begin{tikzpicture}[
    grow = south,
    level 1/.style = {sibling distance = 36mm, level distance = 10mm},
    level 2/.style = {sibling distance = 20mm, level distance = 10mm},
    level 3/.style = {sibling distance = 16mm, level distance = 10mm},
    edge from parent/.style = {thick, draw = black!50, shorten >= 1mm, shorten <= 1mm},
    every node/.style = {rectangle, rounded corners, shade, top color = white,
      bottom color = black!20, draw = black!50, thick, inner sep = 0mm},
    mynodestyle/.style = {minimum width = 12mm, minimum height = 6mm},
    myleafstyle/.style = {minimum width = 7mm, minimum height = 5mm}
    ]
    \node[mynodestyle]{$V_0$}
    child[edge from parent/.style = {thick, draw = black!50, shorten >= 1mm, shorten <= 1mm}]{
      node[myleafstyle]{$E_1$}
    }
    child[edge from parent/.style = {thick, draw = black!50, shorten >= 1mm, shorten <= 1mm}]{
      node[mynodestyle]{$V_{01}$}{
        child[edge from parent/.style = {thick, draw = black!50, shorten >= 1mm, shorten <= 1mm}]{
          node[myleafstyle]{$E_2$}
        }
        child[edge from parent/.style = {thick, draw = black!50, shorten >= 1mm, shorten <= 1mm}]{
          node[myleafstyle]{$E_3$}
        }
      }
    }
    child[edge from parent/.style = {thick, draw = black!50, shorten >= 1mm, shorten <= 1mm}]{
      node[mynodestyle]{$V_{02}$}{
        child[edge from parent/.style = {thick, draw = black!50, shorten >= 1mm, shorten <= 1mm}]{
          node[myleafstyle]{$E_4$}
        }
        child[edge from parent/.style = {thick, draw = black!50, shorten >= 1mm, shorten <= 1mm}]{
          node[mynodestyle]{$V_{23}$}{
            child[edge from parent/.style = {thick, draw = black!50, shorten >= 1mm, shorten <= 1mm}]{
              node[myleafstyle]{$E_5$}
            }
            child[edge from parent/.style = {thick, draw = black!50, shorten >= 1mm, shorten <= 1mm}]{
              node[myleafstyle]{$E_6$}
            }
            child[edge from parent/.style = {thick, draw = black!50, shorten >= 1mm, shorten <= 1mm}]{
              node[myleafstyle]{$E_7$}
            }
          }
        }
      }
    };
  \end{tikzpicture}
  \caption{Tree representation of hierarchical frailties for the construction of a HAXC.}
  \label{fig:NAC:tree}
\end{figure}
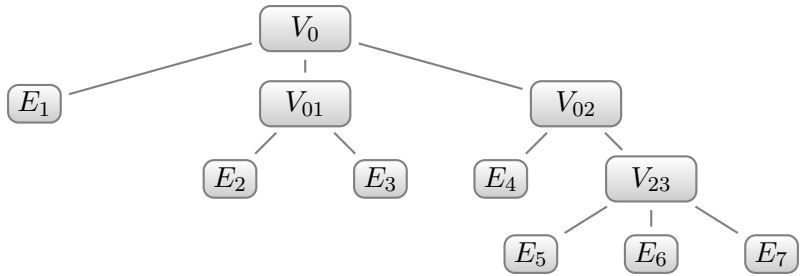
The hierarchical frailties are shown as nodes and the corresponding (dependent) $\Exp(1)$ random variables
as leaves. The frailty at each level is drawn from a distribution on the
positive real line which depends on the frailty from one level before: First
$V_0\sim F_0$ is drawn; then, independently of each other,
$V_{01}\sim F_{01}(\cdot;V_0)$ and $V_{02}\sim F_{02}(\cdot;V_0)$ are drawn
(note that $V_0$ thus acts as a parameter on the distributions $F_{01}$ of
$V_{01}$ and $F_{02}$ of $V_{02}$); finally, $V_{23}\sim F_{23}(\cdot;V_{02})$
is drawn. This procedure can easily be generalized (level by level) to more
hierarchical levels if so desired. Similar to the Archimax case, if
$(E_1,\dots,E_7)$ has EVC $D$ and $\Exp(1)$ margins, one considers
\begin{align}
  \biggl(\frac{E_1}{V_0},\frac{E_2}{V_{01}},\frac{E_3}{V_{01}},\frac{E_4}{V_{02}},\frac{E_5}{V_{23}},\frac{E_6}{V_{23}},\frac{E_7}{V_{23}}\biggr)\label{eq:HAXC:X}
\end{align}
and the survival copula of this random vector is then the HAXC
$C$. For the latter step one needs the marginal survival functions of this
random vector which are typically not known explicitly. However, they are known
under the so-called \emph{sufficient nesting condition} which is based on
certain Laplace--Stieltjes transforms involved and which is also utilized in the
construction of nested Archimedean copulas (NACs); see, for example,
\cite[pp.~87]{joe1997}, \cite{mcneil2008} or \cite{hofert2011a}. To introduce
these Laplace--Stieltjes transforms, it is convenient to have the construction
principle of NACs in mind. The NAC corresponding to Figure~\ref{fig:NAC:tree} is
given by $C_0\bigl(u_1, C_1(u_2, u_3), C_2(u_4, C_3(u_5, u_6, u_7))\bigr)$,
where $C_k$ is generated by the completely monotone generator $\psi_k$,
$k\in\{0,1,2,3\}$. For this case, the sufficient nesting condition requires the
appearing \emph{nodes} $\psi_0^{-1}\circ\psi_1$, $\psi_0^{-1}\circ\psi_2$ and
$\psi_2^{-1}\circ\psi_3$ in NAC to have completely monotone derivatives; see
\cite{hofert2010c} for examples and general results when this holds.
This implies that the functions
$\psi_{kl}(t;v)=\exp\bigl(-v\psi_k^{-1}(\psi_l(t))\bigr)$, $t \in [0,\infty)$,
$v\in(0,\infty)$, for $(k,l)=(0,1)$, $(k,l)=(0,2)$ and $(k,l)=(2,3)$ are
completely monotone generators for every $v$; see \cite[p.~441]{feller1971}. As
such, by Bernstein's Theorem, they correspond to distribution functions on the
positive real line. The important part now is that if the frailties $V_0$,
$V_{01}$, $V_{02}$ and $V_{23}$ are chosen level-by-level such that
\begin{enumerate}
\item $V_0\sim F_0=\LSi[\psi_0]$;
\item $V_{01}\,|\,V_0\sim F_{01}=\LSi[\psi_{01}(\cdot\,;V_0)]$ and $V_{02}\,|\,V_0\sim
  F_{02}=\LSi[\psi_{02}(\cdot\,;V_0)]$; and
\item $V_{23}\,|\,V_{02}\sim F_{23}=\LSi[\psi_{23}(\cdot\,;V_{02})]$,
\end{enumerate}
then, by following along the lines as described in \cite{hofert2012b}, one can show that the corresponding
HAXC has the stochastic representation
\begin{align}
  \bm{U}=\biggl(\psi_0\biggl(\frac{E_1}{V_0}\biggr),
  \psi_1\biggl(\frac{E_2}{V_{01}}\biggr), \psi_1\biggl(\frac{E_3}{V_{01}}\biggr),
  \psi_2\biggl(\frac{E_4}{V_{02}}\biggr), \psi_3\biggl(\frac{E_5}{V_{23}}\biggr),
  \psi_3\biggl(\frac{E_6}{V_{23}}\biggr), \psi_3\biggl(\frac{E_7}{V_{23}}\biggr)\biggr).\label{eq:HAXC:stochrep}
\end{align}
By comparison with \eqref{eq:HAXC:X}, we see that if the distribution functions
$F_0$, $F_{01}$, $F_{02}$, $F_{23}$ of $V_0\sim F_0$,
$V_{01}\sim F_{01}(\cdot;V_0)$, $V_{02}\sim F_{02}(\cdot;V_0)$,
$V_{23}\sim F_{23}(\cdot;V_{02})$ are chosen such that the Laplace--Stieltjes
transforms $\psi_0$, $\psi_1$, $\psi_2$, $\psi_3$ (associated to
$V_0,V_{01},V_{02},V_{23}$ via the structure of a NAC) satisfy the sufficient
nesting condition, then the marginal survival functions of \eqref{eq:HAXC:X} are
not only known, but they are equal to $\psi_0,\psi_1,\psi_2,\psi_3$ such that
the resulting HAXC has a stochastic representation (see
\eqref{eq:HAXC:stochrep}) similar to that of a HAXC with single fraility (see
\eqref{eq:stoch:rep:AXC}), just with different frailties.

\begin{remark}
  \begin{enumerate}
  \item Clearly, the stochastic representation of a HAXC based on hierarchical
    frailties as in \eqref{eq:HAXC:stochrep} immediately allows for a sampling
    algorithm.  The hierarchical frailties involved can easily be sampled in
    many cases, see \cite{hofert2010c} or the \R\ package \texttt{copula} of
    \cite{copula} for details.
  \item Note that the stochastic representation of a HAXC constructed with
    hierarchical frailties equals that of a NAC, except for the fact that for
    the latter, the EVC $D$ of $(E_1,\dots,E_7)$ is the independence
    copula.
  \item The two types of constructing HAXCs presented here can also be mixed,
    one can use a HEVC and hierarchical frailties. Interestingly, the two types
    of hierarchies introduced this way do not have to coincide; see the
    following section for such an example.
  \end{enumerate}
\end{remark}

The figures shown in the following examples can all be reproduced with the vignette \texttt{HAXC} of the \R\ package \texttt{copula}.
\begin{example}[ACs vs AXCs vs NACs vs (different) HAXCs]\label{ex:splom:1}
  Figure~\ref{fig:scatter:plots} shows scatter-plot matrices of five-dimensional copula
  samples of size 1000 from the following models for $\bm{U}=(U_1,\dots,U_5)\sim
  C$.
  \begin{enumerate}
  \item\label{ex:splom:1:1} Top left: (Archimedean) Clayton copula with stochastic representation
    \begin{align}
      \bm{U}=\biggl(\psi\biggl(\frac{E_1}{V}\biggr),\dots,\psi\biggl(\frac{E_5}{V}\biggr)\biggr),\label{eq:ex:stoch:rep:AC}
    \end{align}
    where $V\sim\Gamma(1/\theta,1)$ for $\theta=4/3$ (the frailty is gamma
    distributed) and $E_1,\dots,E_5\isim\Exp(1)$; see also
    \eqref{eq:stoch:rep:AC}. The copula parameter is chosen such that Kendall's
    tau equals 0.4.
  \item Top right: AXC based on Clayton's family with gamma frailties recycled from
    the top left plot and stochastic representation as in
    \eqref{eq:ex:stoch:rep:AC} where
    $(E_1,\dots,E_5)=(-\log Y_1,\dots,-\log Y_5)$ for $(Y_1,\dots,Y_5)$ having
    a Gumbel EVC (with parameter such that Kendall's tau equals 0.5); note that
    the margins of $(E_1,\dots,E_5)$ are again $\Exp(1)$ (but its components are
    dependent in this case).
  \item\label{ex:splom:1:3} Middle left: NAC based on Clayton's family with hierarchical frailties such that two sectors of sizes 2 and 3 result,
    respectively, with parameters ($\theta_0,\theta_1,\theta_2$) chosen such that Kendall's tau equals 0.2
    between the two sectors, 0.4 within the first sector and 0.6 within the second
    sector. A stochastic representation for this copula is given by
    \begin{align}
      \bm{U}=\biggl(\psi_1\biggl(\frac{E_1}{V_{01}}\biggr),\psi_1\biggl(\frac{E_2}{V_{01}}\biggr),\psi_2\biggl(\frac{E_3}{V_{02}}\biggr),
      \psi_2\biggl(\frac{E_4}{V_{02}}\biggr), \psi_2\biggl(\frac{E_5}{V_{02}}\biggr)\biggr),\label{eq:ex:stoch:rep:NAC}
    \end{align}
    where $V_0\sim\Gamma(2)$ and $V_{01}\,|\,V_0\sim
    F_{01}=\LSi\bigl[\exp\bigl(-V_0((1+t)^{\theta_0/\theta_1}-1)\bigr)\bigr]$,
    $V_{02}\,|\,V_0\sim
    F_{02}=\LSi\bigl[\exp\bigl(-V_0((1+t)^{\theta_0/\theta_2}-1)\bigr)\bigr]$
    are independent (see \cite[Theorem~3.6]{hofert2011a} for more details) and
    $E_1,\dots,E_5\isim\Exp(1)$.
  \item Middle right: HAXC based on Clayton's family with hierarchical
    frailties recycled from the middle left plot and stochastic representation
    as in \eqref{eq:ex:stoch:rep:NAC} where
    $(E_1,\dots,E_5)=(-\log Y_1,\dots,-\log Y_5)$ for $(Y_1,\dots,Y_5)$ having
    a Gumbel EVC (realizations recycled from the top right plot). Note that
    the hierarchical structure is only induced by the frailties in this case.
  \item Bottom left: HAXC based on Clayton's family with hierarchical frailties recycled from the middle left plot and
    stochastic representation
    \begin{align*}
      \bm{U}=\biggl(\psi_1\biggl(\frac{E_{11}}{V_{01}}\biggr),\psi_1\biggl(\frac{E_{12}}{V_{01}}\biggr),\psi_2\biggl(\frac{E_{21}}{V_{02}}\biggr),
      \psi_2\biggl(\frac{E_{22}}{V_{02}}\biggr), \psi_2\biggl(\frac{E_{23}}{V_{02}}\biggr)\biggr),
    \end{align*}
    where
    $(E_{11}, E_{12}, E_{21}, E_{22}, E_{23})=(-\log Y_{11},-\log Y_{12},-\log
    Y_{21},-\log Y_{22},-\log Y_{23})$ for
    $(Y_{11},Y_{12},Y_{21},Y_{22},Y_{23})$ having a nested Gumbel EVC (with
    sector sizes 2 and 3 and parameters such that Kendall's tau equals 0.2
    between the two sectors, 0.5 within the first sector and 0.7 within the
    second sector). Note that the hierarchical structure is induced both at the level of the
    frailties and at the level of the EVC in this case, and that the hierarchical structure
    (sectors, sector dimensions) is the same.
  \item Bottom right: HAXC as in the bottom left plot (realizations recycled) with stochastic representation
    \begin{align*}
      \bm{U}=\biggl(\psi_1\biggl(\frac{E_{11}}{V_{01}}\biggr),\psi_1\biggl(\frac{E_{12}}{V_{01}}\biggr),\psi_1\biggl(\frac{E_{21}}{V_{01}}\biggr),
      \psi_2\biggl(\frac{E_{22}}{V_{02}}\biggr), \psi_2\biggl(\frac{E_{23}}{V_{02}}\biggr)\biggr).
    \end{align*}
    Note that the hierarchical structure for the frailties (sector sizes 3 and 2, respectively) and for the
    nested Gumbel EVC (sector sizes 2 and 3, respectively) differ in this case.
  \end{enumerate}
  \begin{figure}[htbp]
    \centering
    \includegraphics[width=0.37\textwidth]{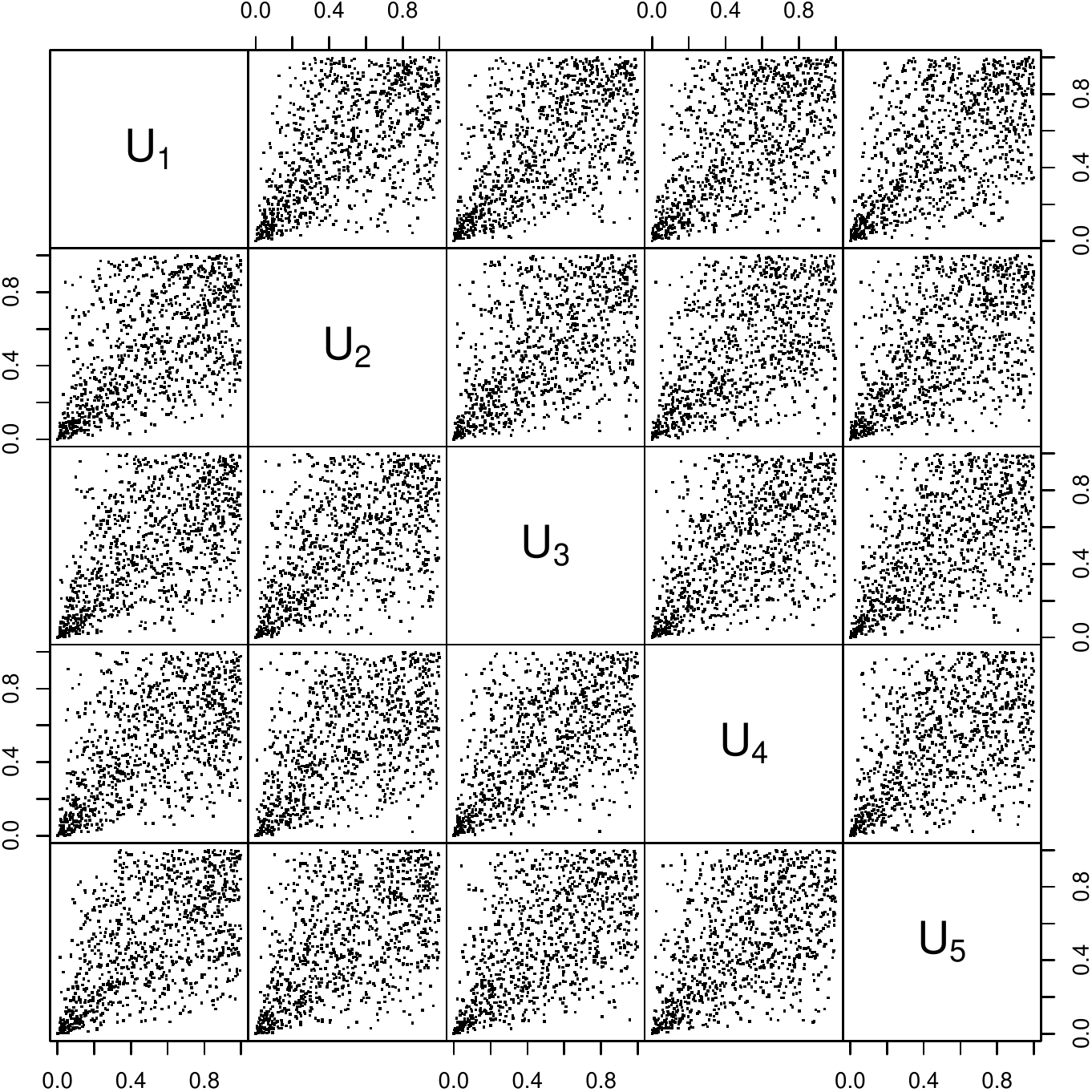}%
    \hspace{6mm}
    \includegraphics[width=0.37\textwidth]{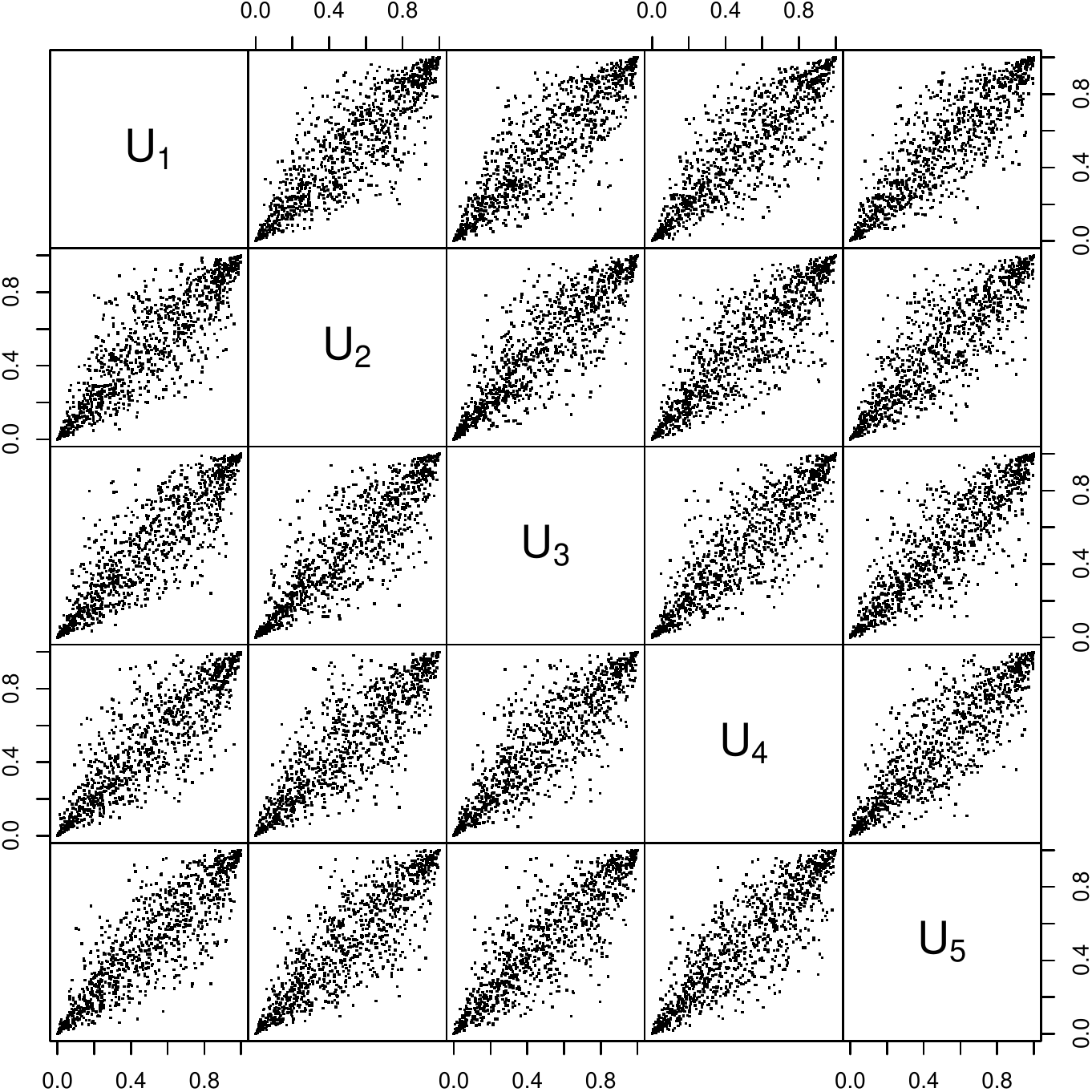}%
    \par\medskip
    \includegraphics[width=0.37\textwidth]{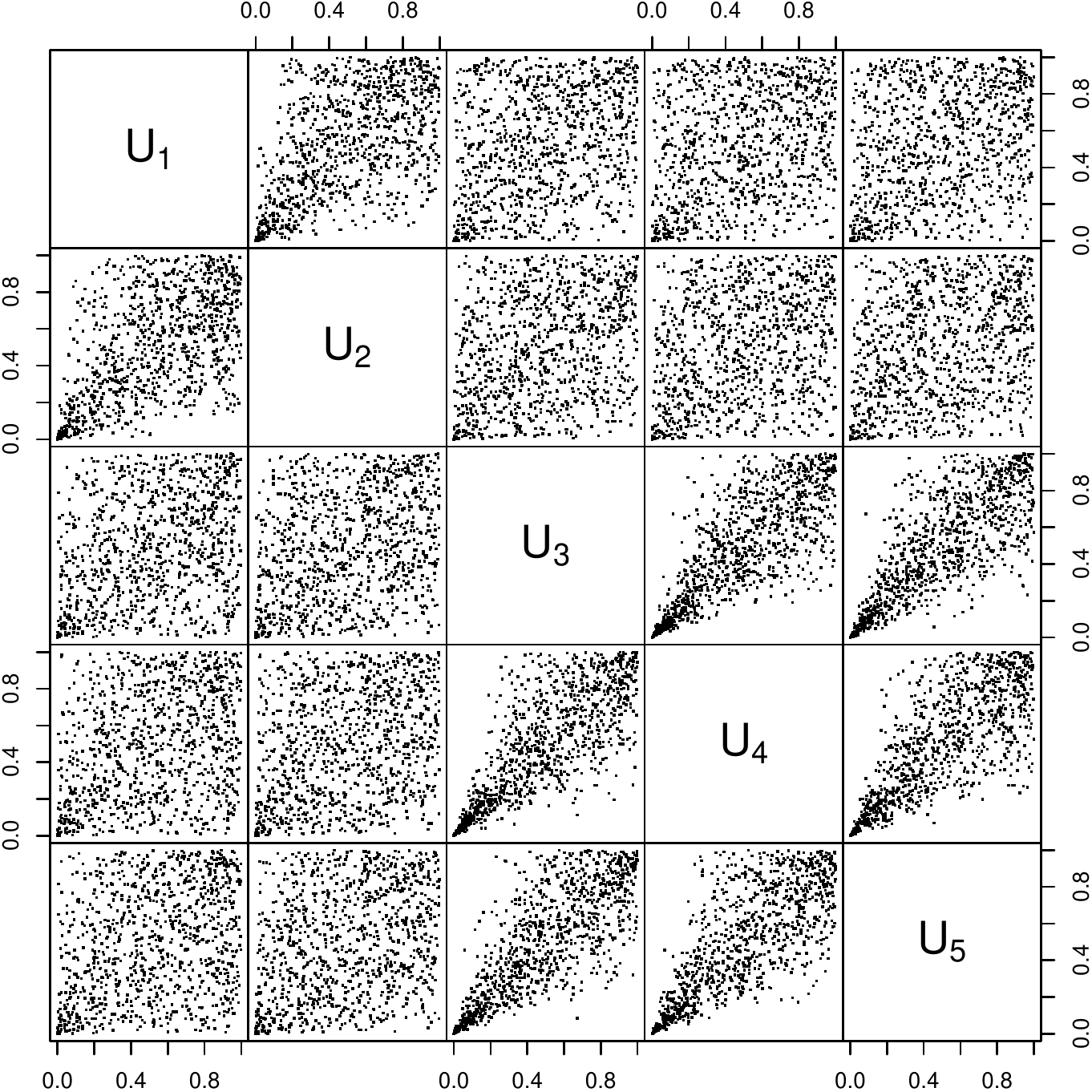}%
    \hspace{6mm}
    \includegraphics[width=0.37\textwidth]{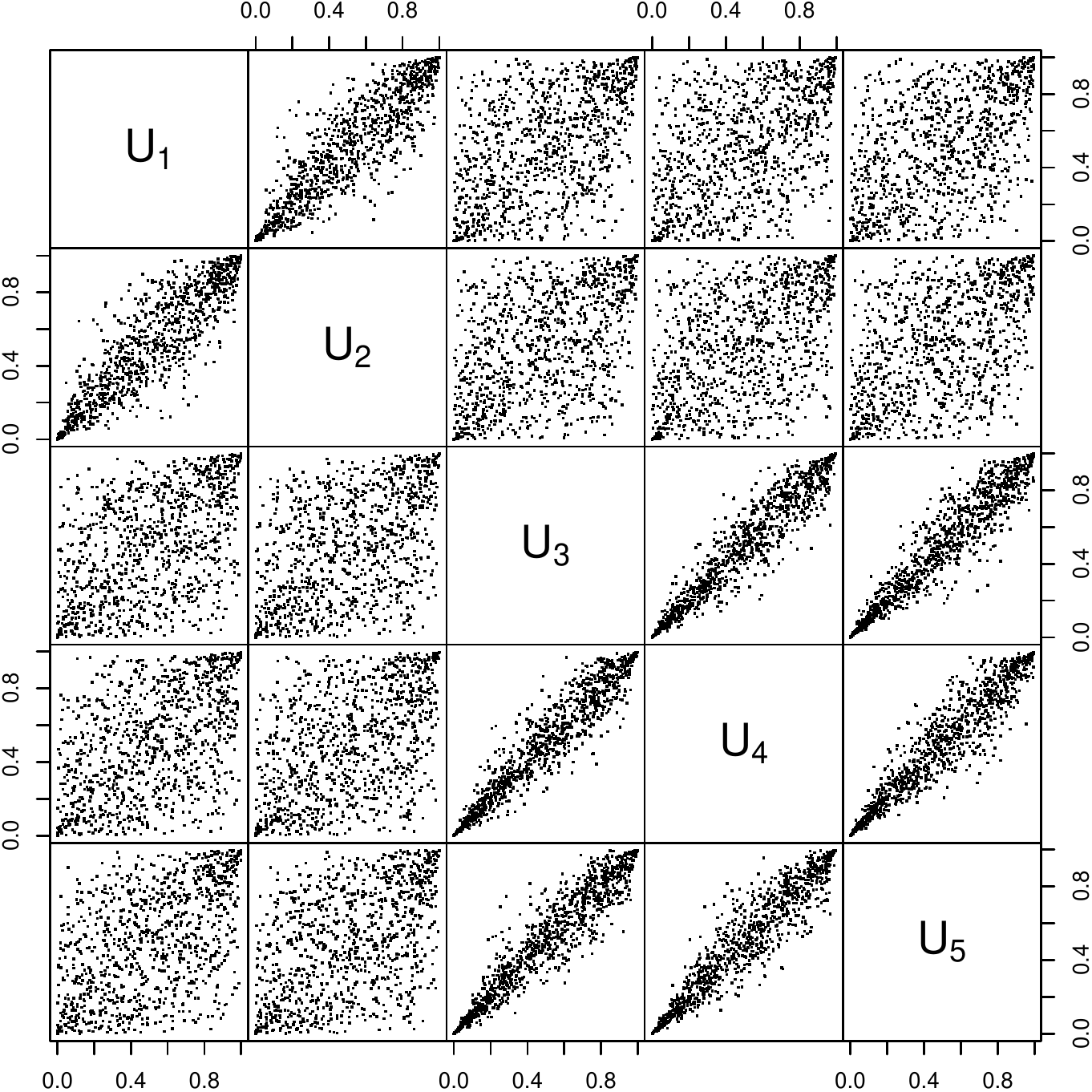}%
    \par\medskip
    \includegraphics[width=0.37\textwidth]{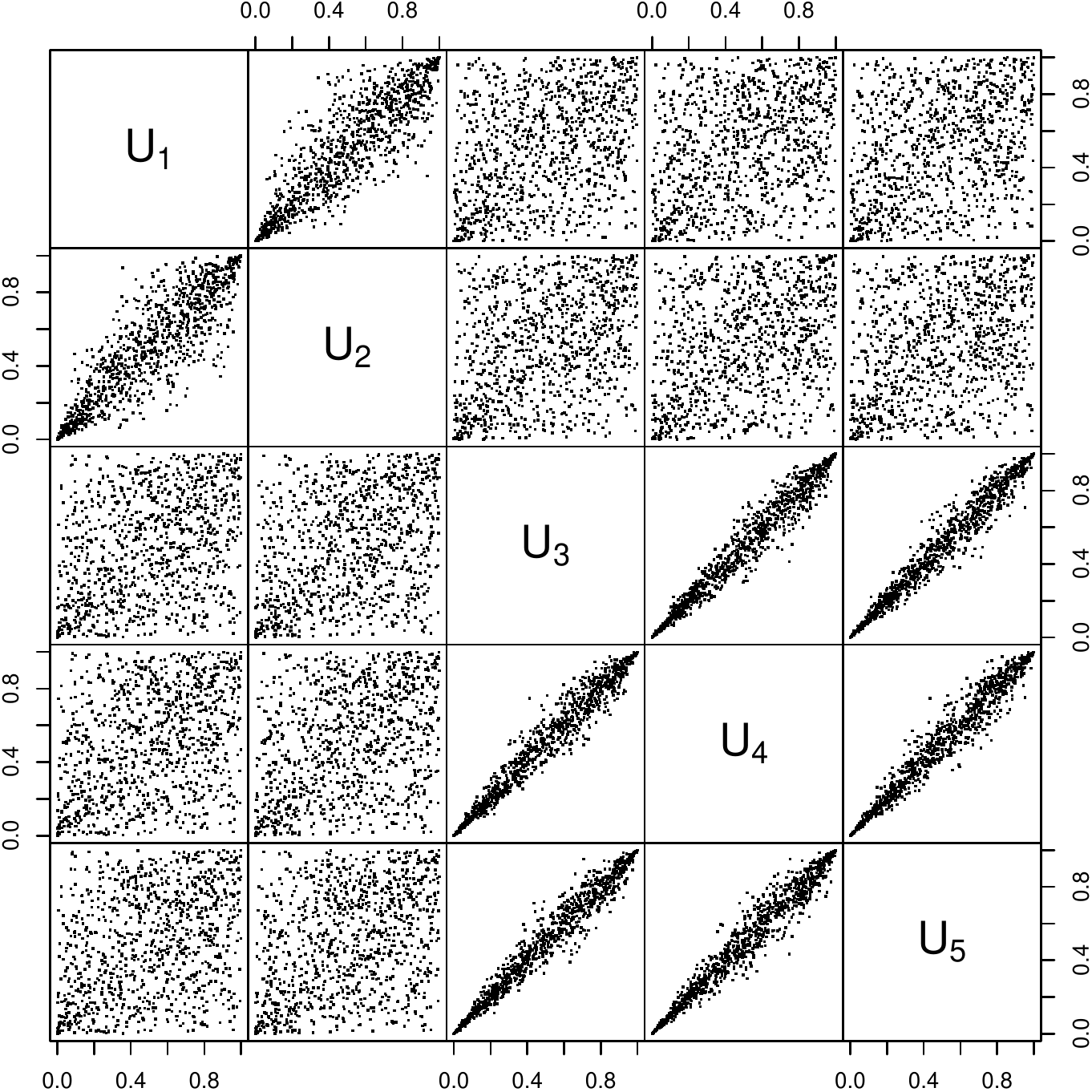}%
    \hspace{6mm}
    \includegraphics[width=0.37\textwidth]{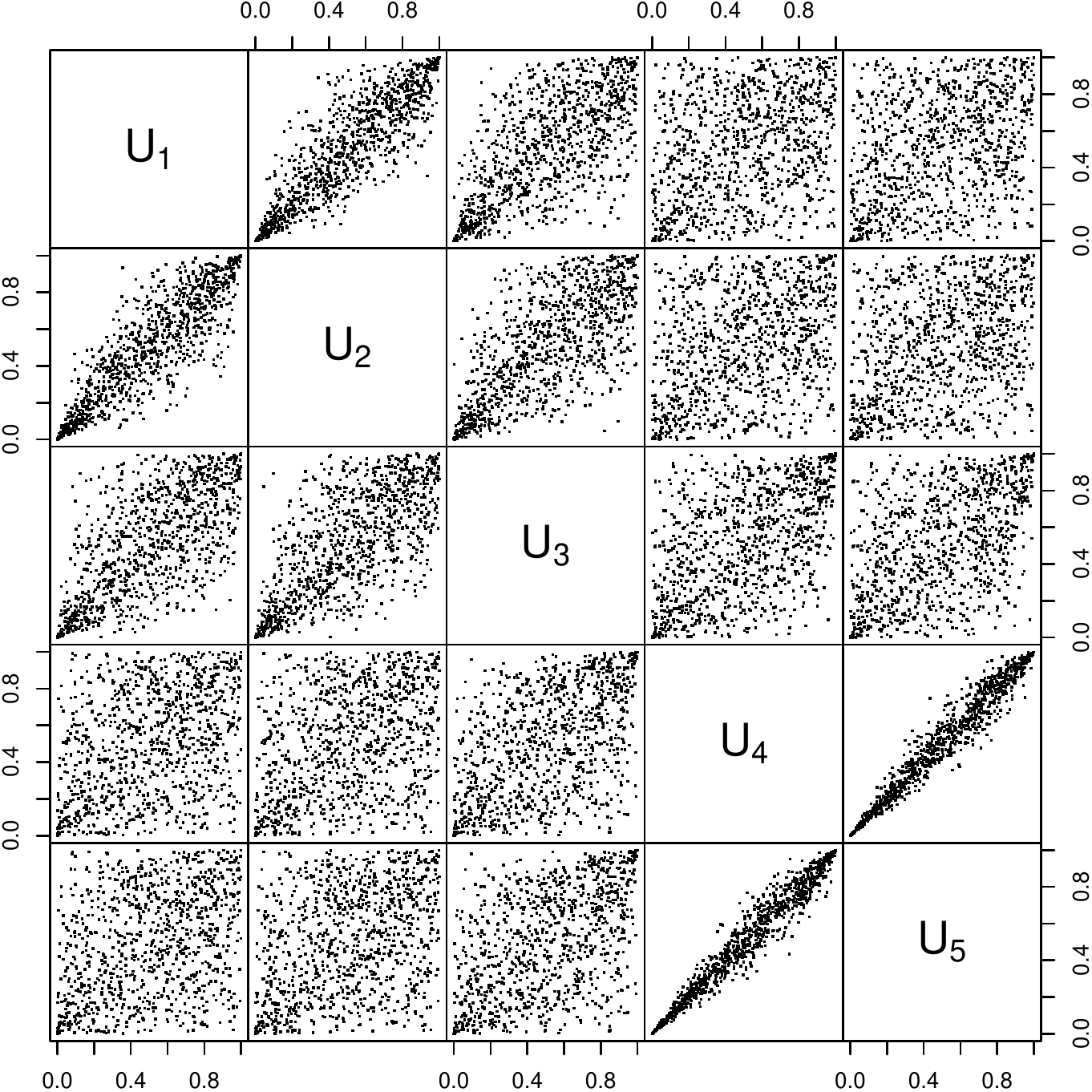}%
    \caption{Scatter-plot matrices of five-dimensional copula samples of size 1000 of a
      Clayton copula (top left), an AXC with Clayton frailties and Gumbel EVC
      (top right), a nested Clayton copula (middle left), a HAXC with hierarchical
      Clayton frailties and Gumbel EVC (middle right), a HAXC with hierarchical
      Clayton frailties and nested Gumbel EVC of the same hierarchical structure
      (bottom left) and a HAXC with hierarchical Clayton frailties and nested Gumbel
      EVC of different hierarchical structure (bottom right).}
    \label{fig:scatter:plots}
  \end{figure}
\end{example}

\begin{example}[EVCs vs HEVCs vs (different) HAXCs] \label{ex:evcvshevc}
  Similar to Figure~\ref{fig:scatter:plots}, Figure~\ref{fig:scatter:plots:2}
  shows scatter-plot matrices of five-dimensional copula samples of size 1000
  from the following models for $\bm{U}=(U_1,\dots,U_5)\sim C$; for simulating from the extremal t EVC,
  we use the \R\ package \texttt{mev} of \cite{mev}.
  \begin{enumerate}
  \item Top left: Extremal $t$ EVC with $\nu=3.5$ degrees of freedom
    and homogeneous correlation matrix $P$ with off-diagonal entries 0.7.
  \item Top right: Extremal $t$ HEVC with two sectors of sizes 2
    and 3, respectively, such that the correlation matrix $P$ has entries 0.2 for pairs
    belonging to different sectors, 0.5 for pairs belonging to the first sector
    and 0.7 for pairs belonging to the second sector.
  \item Middle left: HAXC with single Clayton frailty (as in
    Example~\ref{ex:splom:1} Part~\ref{ex:splom:1:1}) and
    extremal $t$ HEVC recycled from the top right plot.
  \item Middle right: HAXC with hierarchical Clayton frailties (as in Example~\ref{ex:splom:1} Part~\ref{ex:splom:1:3})
    and extremal $t$ EVC recycled from the top left plot.
  \item Bottom left: HAXC with hierarchical Clayton frailties (as in
    Example~\ref{ex:splom:1} Part~\ref{ex:splom:1:3}) and extremal $t$ HEVC
    recycled from the top right plot. Note that there are two types of hierarchies
    involved, at the level of the (hierarchical) frailties and at the level of the
    (hierarchical) extremal $t$ EVC. Furthermore, the two hierarchical structures match.
  \item Bottom right: HAXC as in the bottom left plot, but the hierarchical
    structures of the frailties (sector sizes 3 and 2, respectively) and of the
    HEVC (sector sizes 2 and 3, respectively) differ in this case.
  \end{enumerate}
  Note that we can sample from a hierarchical Schlather model (special case of
  extremal $t$ for $\nu=1$), a hierarchical Brown--Resnick model, and their
  corresponding HAXCs in a similar fashion.
  \begin{figure}[htbp]
    \centering
    \includegraphics[width=0.37\textwidth]{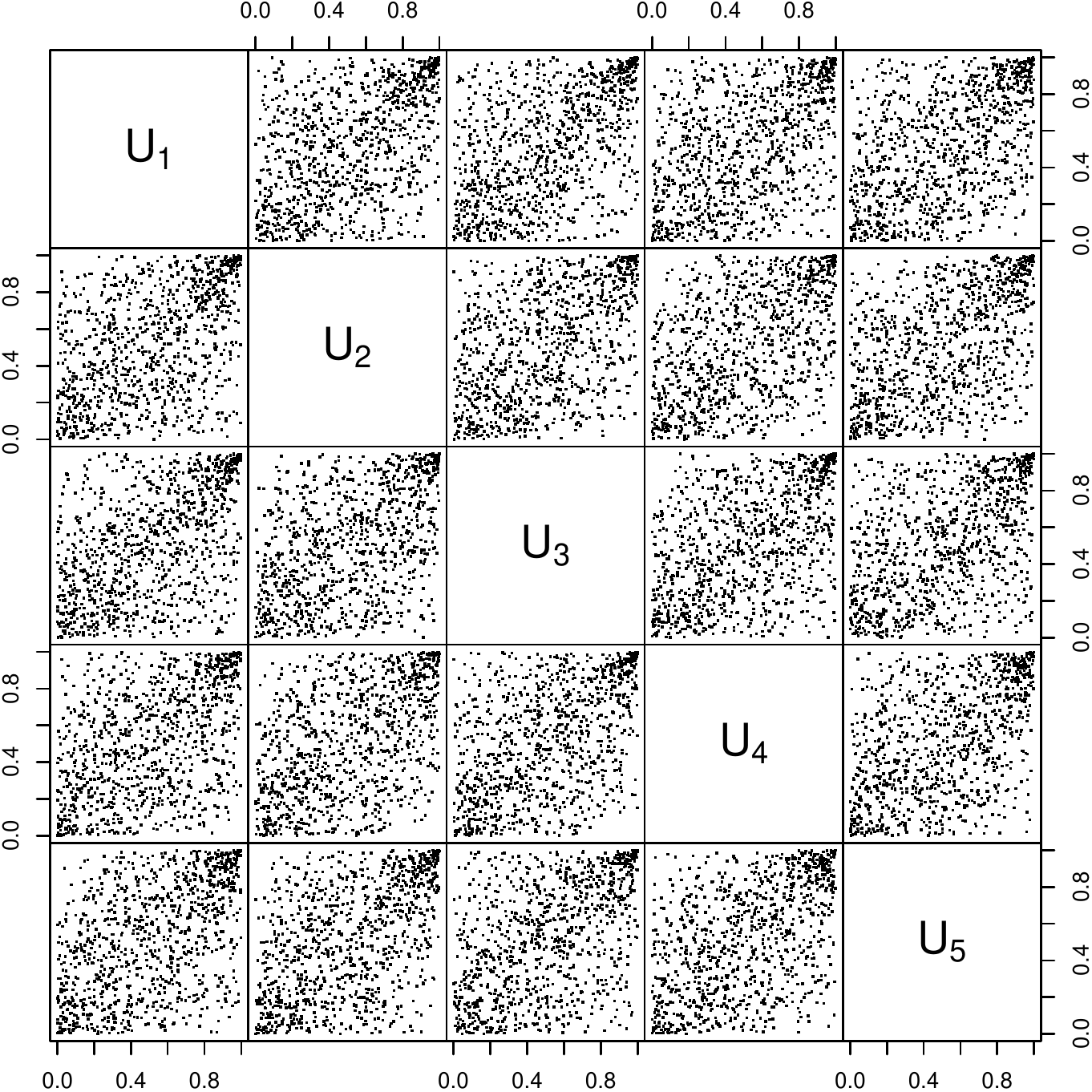}%
    \hspace{6mm}
    \includegraphics[width=0.37\textwidth]{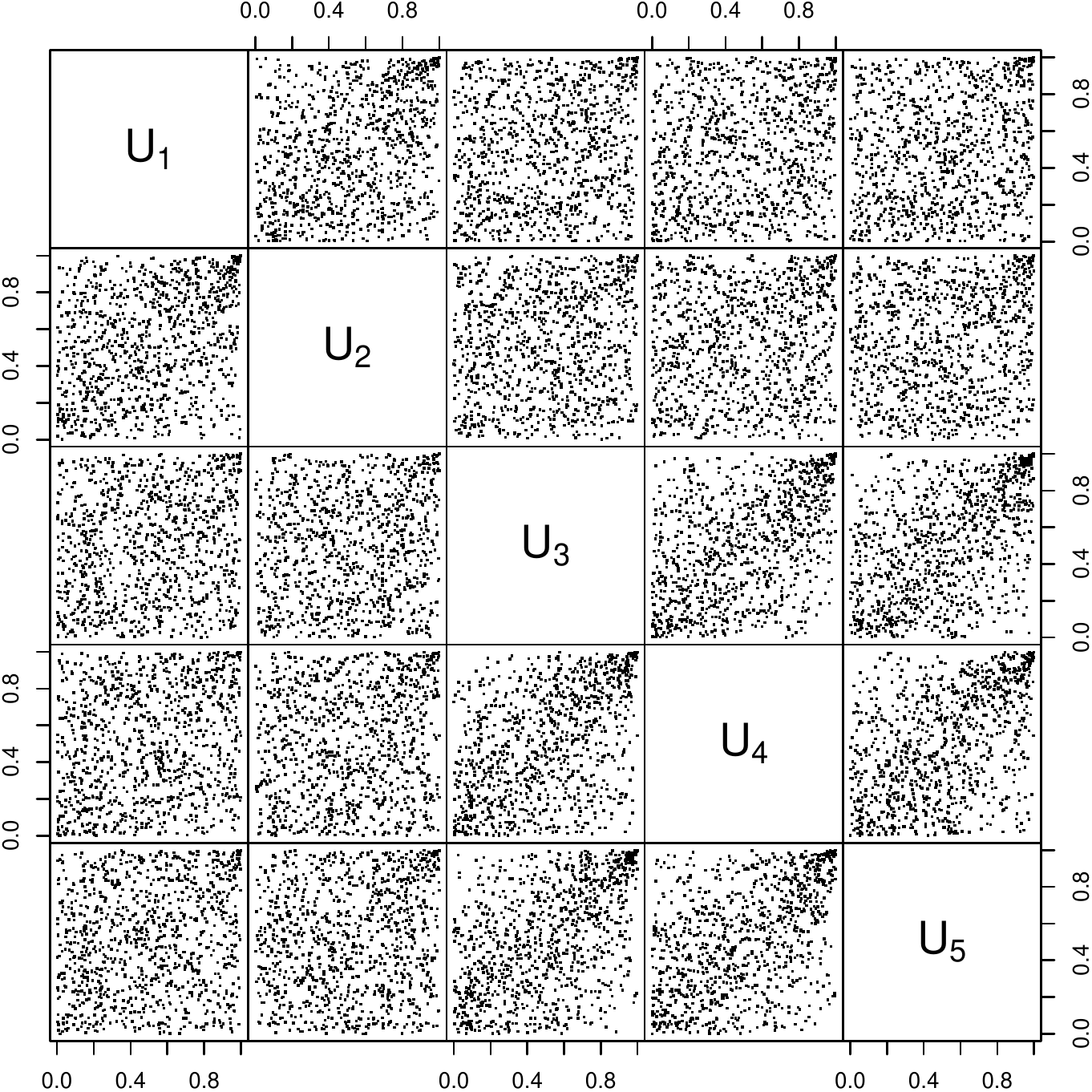}%
    \par\medskip
    \includegraphics[width=0.37\textwidth]{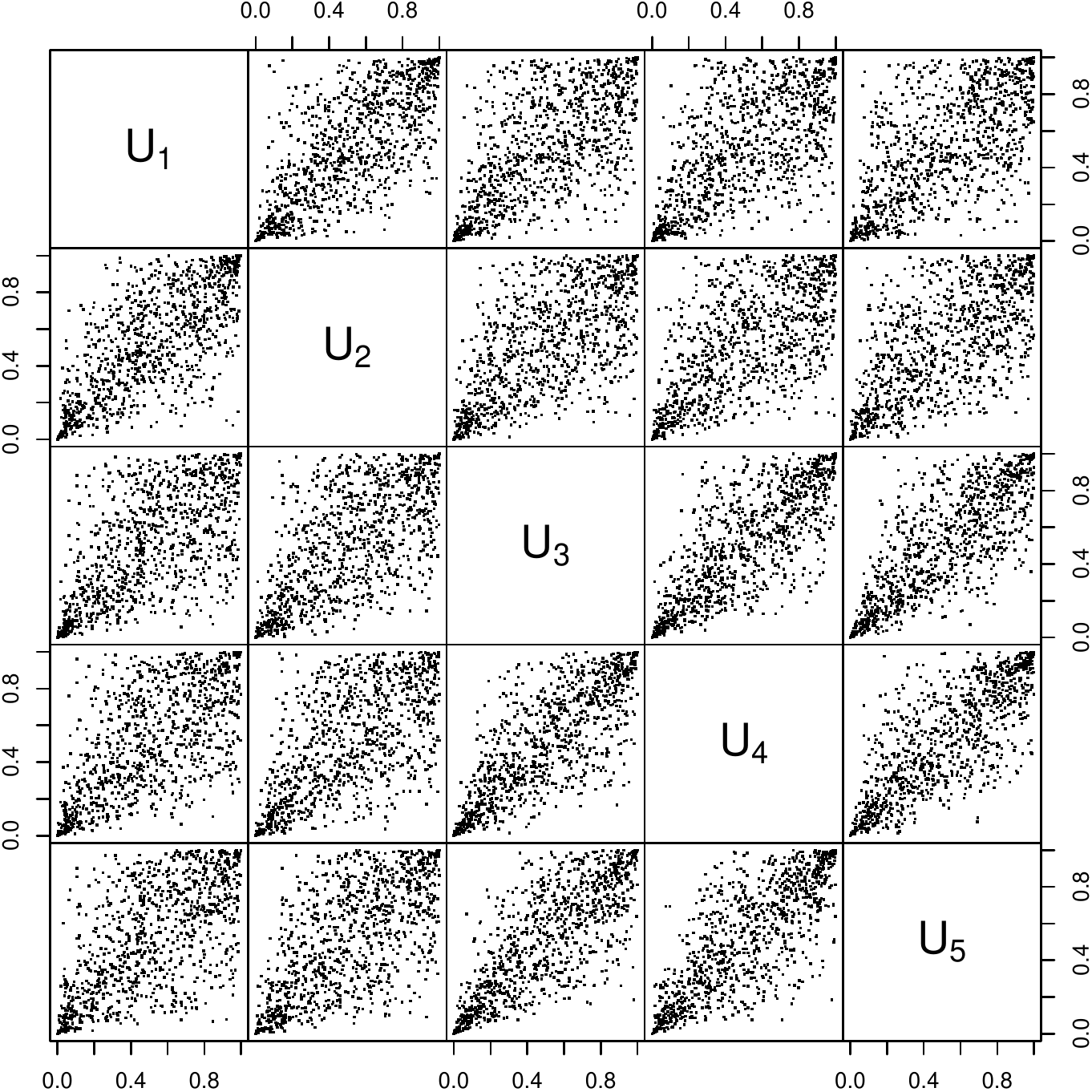}%
    \hspace{6mm}
    \includegraphics[width=0.37\textwidth]{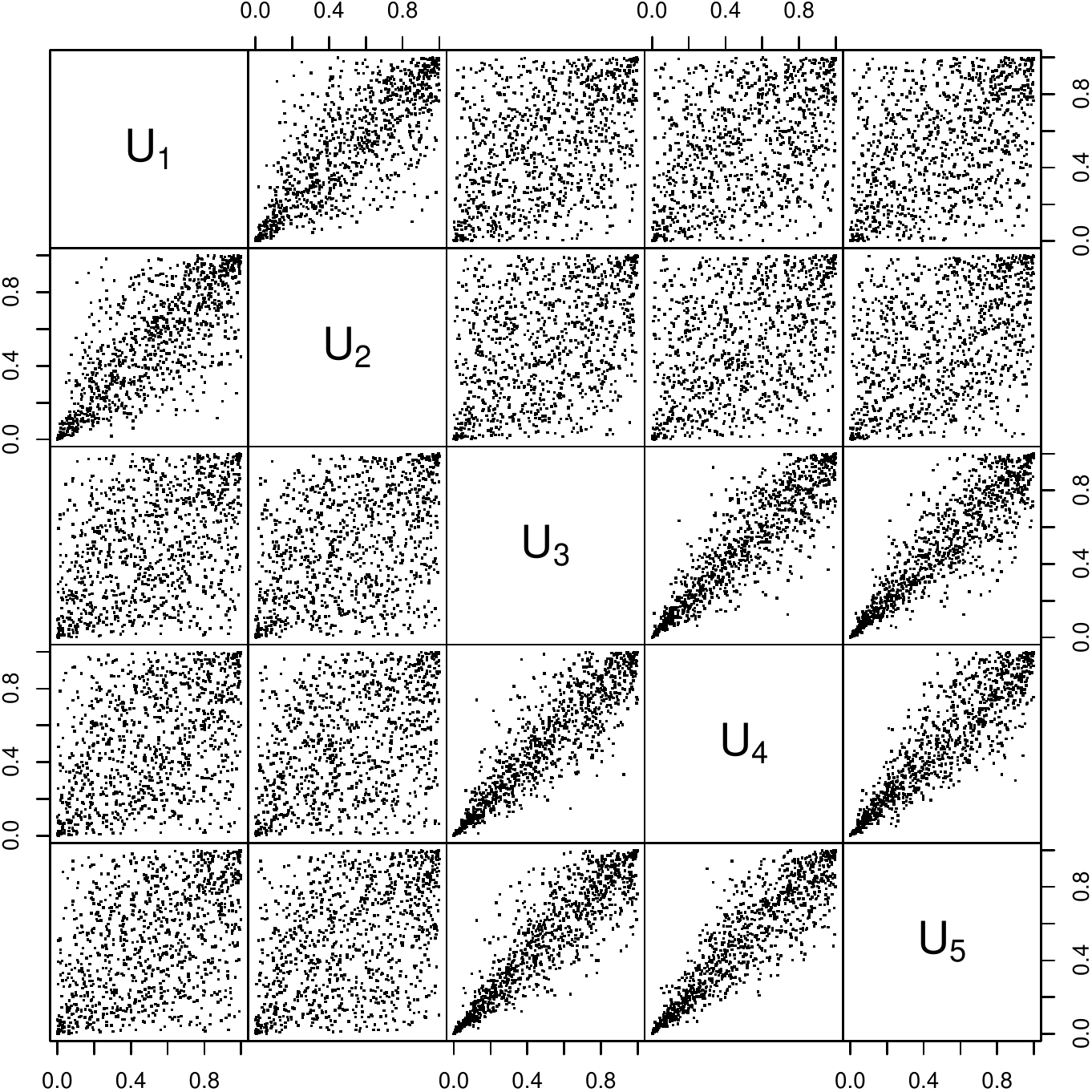}%
    \par\medskip
    \includegraphics[width=0.37\textwidth]{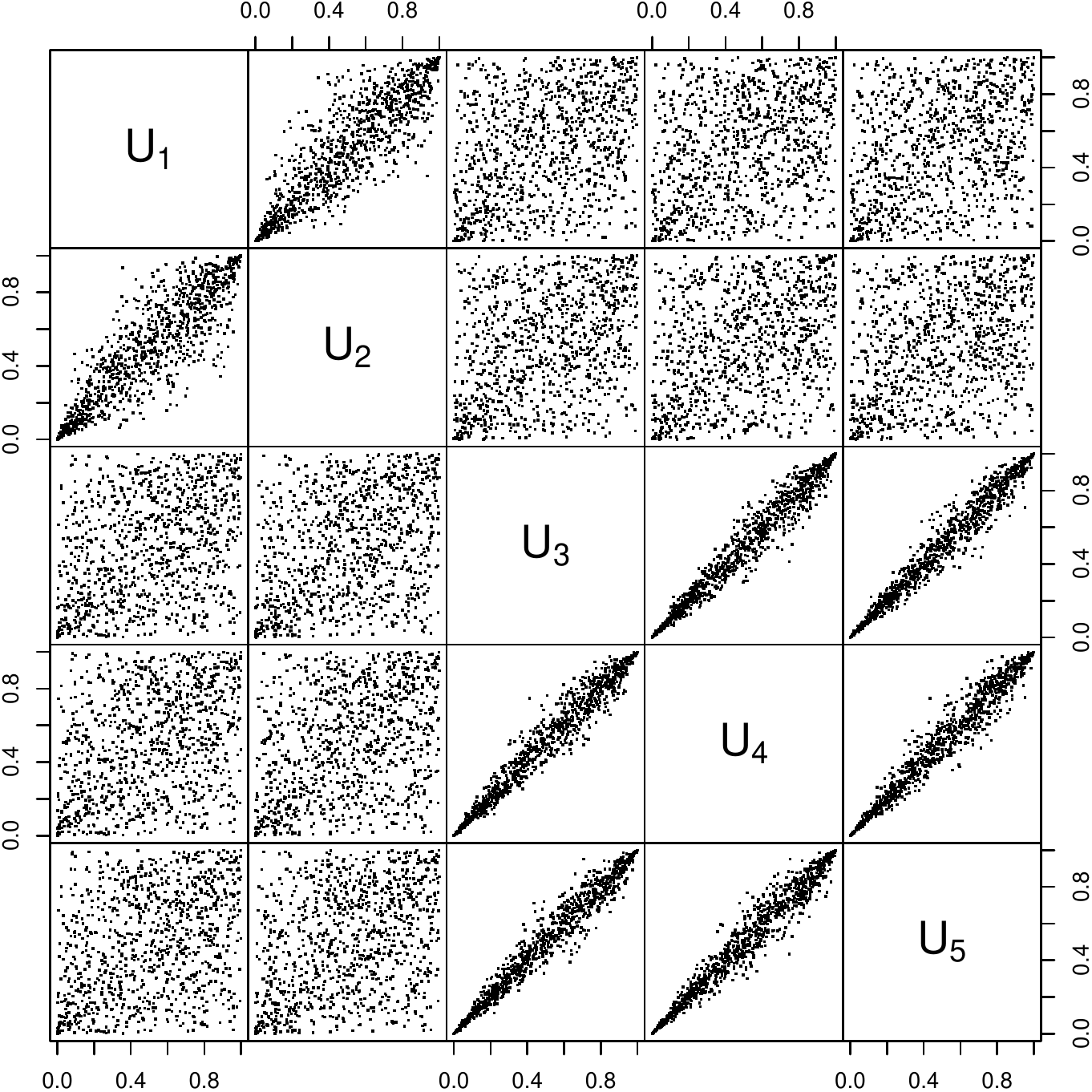}%
    \hspace{6mm}
    \includegraphics[width=0.37\textwidth]{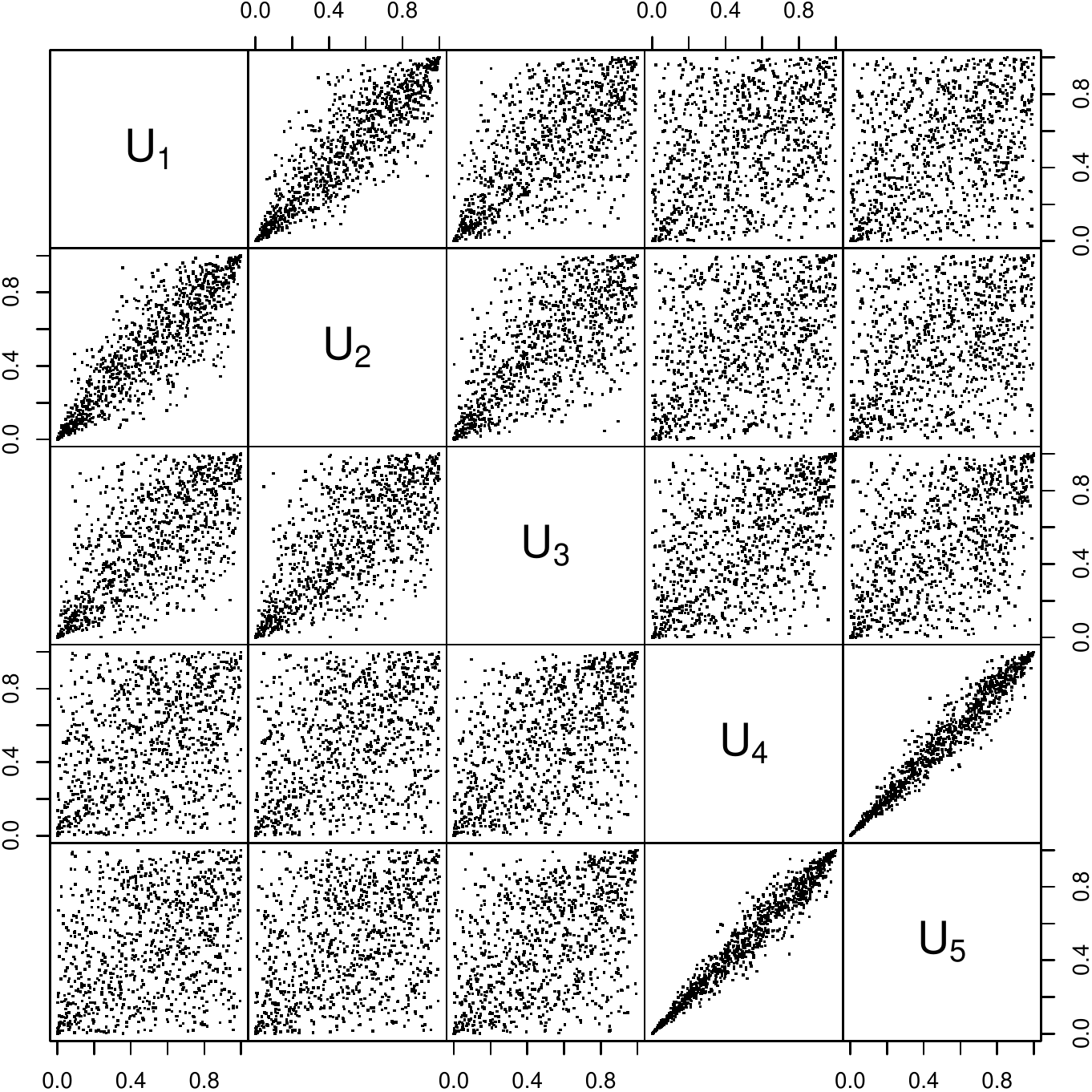}%
    \caption{Scatter-plot matrices of five-dimensional copula samples of size
      1000 of an extremal $t$ EVC (top left), a hierarchical extremal $t$
      copula (a HEVC; top right), a HAXC with single Clayton frailty and
      extremal $t$ HEVC (middle left), a HAXC with hierarchical Clayton frailties
      and extremal $t$ EVC (middle right), a HAXC with hierarchical Clayton frailties
      and extremal $t$ HEVC of the same hierarchical structure (bottom left) and a HAXC
      with hierarchical Clayton frailties and extremal $t$ HEVC of different hierarchical
      structure (bottom right).}
    \label{fig:scatter:plots:2}
  \end{figure}
\end{example}

\section{Conclusion}
We extended the class of AXCs to HAXCs. Hierarchies can take place in two forms,
either separately or simultaneously. First, the EVC involved in the construction
of AXCs can have a hierarchical structure. To this end we presented a new
approach for constructing hierarchical stable tail dependence functions based on
a connection between stable tail dependence functions and $d$-norms. Second, a
hierarchical structure can be imposed at the level of frailties similarly as
NACs arise from ACs. Even more flexible constructions can be obtained by
choosing a different hierarchical structure for the HEVC and the
hierarchical frailties in the construction. Since all presented constructions
are based on stochastic representations, sampling is immediate; see also the
presented examples and vignette.

As a contribution to the literature on AXCs, we also derived a general formula
for the density of AXCs and the computation of the corresponding logarithmic
density. Furthermore, we briefly addressed the question when nested AXCs (NAXCs)
can be constructed (either through nested stable tail dependence functions alone
or, additionally, through hierarchical frailties). This is, in principle,
possible, but there is currently only one family of examples known when all the
assumptions involved are fulfilled. Further research is thus required to find
out whether this is the only possible case for which NAXCs result.

\subsection*{Acknowledgments}
The first author acknowledges support from NSERC (Grant RGPIN-5010-2015) and
FIM, ETH Z\"urich. The third author acknowledges support from NSERC (PGS D
scholarship). We would also like to thank the AE and reviewers for their
comments which helped to improve the paper substantially.

\appendix

\section{Density of Archimax copulas}\label{app:AXC:dens}
For likelihood-based inference based on AXCs, it is important to
know their density. In this section, we present the general form of the density
of AXCs (if it exists) and address how it can be computed numerically.
\begin{proposition}[AXC density]\label{prop:AXC:dens}
  If the respective partial derivatives of $\ell$ exist and are continuous, the
  density $c$ of a $d$-dimensional AXC $C$ is given by
  \begin{align*}
    c(\bm{u})=\biggl(\,\prod_{j=1}^d (\psii)'(u_j)\biggr)\sum_{k=1}^d\psi^{(k)}\bigl(\ell(\psii(\bm{u}))\bigr)\!\!\!\!\!\!\sum_{\pi\in\Pi:|\pi|=k}\prod_{B\in\pi}(\D_B \ell)(\psii(\bm{u})),\quad\bm{u}\in(0,1)^d,
  \end{align*}
  where $\psii(\bm{u})=(\psii(u_1),\dots,\psii(u_d))$, $\Pi$ denotes the set of
  all partitions $\pi$ of $\{1,\dots,d\}$ (with $|\pi|$ denoting the number of
  elements of $\pi$) and $(\D_B\ell)(\psi^{-1}(\bm u))$ denotes the partial
  derivatives of $\ell$ with respect to the variables with index in $B$,
  evaluated at $\psii(\bm{u})$.
\end{proposition}
\begin{proof}
  By a multivariate version of Fa\`a di Bruno's Formula, see \cite{hardy2006}, the $d$th derivative of a composition of two
  functions $f:\IR\to\IR$ and $g:\IR^d\to\IR$ is given by
  \begin{align*}
    \D f(g(\bm x)) &= \sum_{\pi\in\Pi}\biggl(f^{(|\pi|)}(g(\bm x))\prod_{B\in\pi}\D_B g(\bm x)\biggr)=\sum_{k=1}^d\sum_{\pi\in\Pi:|\pi|=k}\biggl(f^{(|\pi|)}(g(\bm x))\prod_{B\in\pi}\D_B g(\bm x)\biggr)\notag\\
    &=\sum_{k=1}^d\sum_{\pi\in\Pi:|\pi|=k}\biggl(f^{(k)}(g(\bm x))\prod_{B\in\pi}\D_B g(\bm x)\biggr)=\sum_{k=1}^df^{(k)}(g(\bm x))\!\!\!\!\!\!\sum_{\pi\in\Pi:|\pi|=k}\prod_{B\in\pi}\D_B g(\bm x),
  \end{align*}
  where $\D = \frac{\partial^{d}}{\partial x_d\dots\partial x_1}$,
  $\D_B = \frac{\partial^{|B|}}{\prod_{j\in B}\partial x_j}$, and $B\in\pi$
  means that $B$ runs through all partition elements of $\pi$. Assuming that the
  appearing derivatives exist and are continuous, we obtain from taking $f(x)=\psi(x)$ and
  $g(\bm{x})=\ell(\psii(\bm{x}))$ that
  \begin{align*}
    c(\bm{u})&=\sum_{k=1}^d\psi^{(k)}\bigl(\ell(\psii(\bm{u}))\bigr)\!\!\!\!\!\!\sum_{\pi\in\Pi:|\pi|=k}\prod_{B\in\pi}\frac{\partial^{|B|}}{\prod_{j\in B}\partial u_j} \ell(\psii(\bm{u}))\\
    &=\biggl(\,\prod_{j=1}^d (\psii)'(u_j)\biggr)\sum_{k=1}^d\psi^{(k)}\bigl(\ell(\psii(\bm{u}))\bigr)\!\!\!\!\!\!\sum_{\pi\in\Pi:|\pi|=k}\prod_{B\in\pi}(\D_B \ell)(\psii(\bm{u})),\quad\bm{u}\in(0,1)^d,
  \end{align*}
  where the last equality holds since the derivatives of all of
  $\psii(u_1),\dots,\psii(u_d)$ (from applying the chain rule) appear in each
  summand of the sum $\sum_{\pi\in\Pi:|\pi|=k}$ and can thus be taken out of
  both summations.
\end{proof}

As a quick check of Proposition~\ref{prop:AXC:dens}, we can recover the density of ACs and EVCs.
\begin{corollary}[AC density as special case]
  For $\ell(\bm{x})=\sum_{j=1}^dx_j$, the density of ACs correctly
  follows from Proposition~\ref{prop:AXC:dens} by noting that
  \begin{align*}
    \sum_{\pi\in\Pi:|\pi|=k}\prod_{B\in\pi}(\D_B
    \ell)(\bm{x})=\sum_{\pi\in\Pi:|\pi|=k}\prod_{B\in\pi}\I_{\{|B|=1\}}=\sum_{\pi\in\Pi:|\pi|=k}\I_{\{|B|=1\,\text{for
    all}\,B\in\pi\}}=\I_{\{k=d\}}.
  \end{align*}
\end{corollary}

\begin{corollary}[EVC density as special case]
  For $\psi(t)=\exp(-t)$, $t\geq0$, the density of EVCs correctly
  follows from Proposition~\ref{prop:AXC:dens} as one has
  \begin{align*}
    c(\bm{u})&=\biggl(\,\prod_{j=1}^d \biggl(-\frac{1}{u_j}\biggr)\biggr)\sum_{k=1}^d\exp\bigl(-\ell(-\log(\bm{u}))\bigr)\!\!\!\!\!\!\sum_{\pi\in\Pi:|\pi|=k}\prod_{B\in\pi}(-(\D_B \ell)(-\log(\bm{u}))),\quad\bm{u}\in(0,1)^d;
  \end{align*}
  see, for example, \cite{doyon2013} or \cite{castrucciohusergenton2016}.
\end{corollary}

The following result provides the general form of the density of AXCs based on
the stable tail dependence function $\ell$ of a Gumbel copula.
\begin{corollary}[Density of AXCs with Gumbel stable tail dependence function as special case]
  For the stable tail dependence function
  $\ell(\bm x) = (x_1^{1/\alpha} + \dots + x_d^{1/\alpha})^{\alpha}$,
  $\bm{x}\in[0,\infty)^d$, of a Gumbel copula with parameter
  $\alpha\in(0,1]$, %
  the density $c$ of an AXC is given by
  \begin{align*}
    c(\bm{u})&=\frac{1}{\alpha^d}\biggl(\,\prod_{j=1}^d (\psii)'(u_j)
    \psii(u_j)^{\frac{1}{\alpha}-1}\biggr)\\
    &\phantom{{}={}}\cdot\sum_{k=1}^d\psi^{(k)}\biggl(\biggl(\,\sum_{j=1}^d\psii(u_j)^{\frac{1}{\alpha}}\biggr)^{\alpha}\biggr) \biggl(\,\sum_{j=1}^d
    \psii(u_j)^{\frac{1}{\alpha}} \biggr)^{\alpha k-d}\!\!\!\!\!\!\!\!\!\!\!\!\!\sum_{\pi\in\Pi:|\pi|=k}\prod_{B\in\pi}(\alpha)_{|B|},\quad \bm{u}\in(0,1)^d,
  \end{align*}
  where $(\alpha)_{|B|}=\prod_{l=0}^{|B|-1}(\alpha - l)$ denotes the falling factorial.
\end{corollary}
\begin{proof}
  For the stable tail dependence function $\ell(\bm x) = (x_1^{1/\alpha} + \dots + x_d^{1/\alpha})^{\alpha}$,
  $\bm{x}\in[0,\infty)^d$, $\alpha\in(0,1]$, one has
  \begin{align*}
    \D_B\ell(\bm{x})=(\alpha)_{|B|}\biggl(\,\sum_{j=1}^dx_j^{1/\alpha}\biggr)^{\alpha-|B|}\biggl(\frac{1}{\alpha}\biggr)^{|B|}\prod_{j\in
    B}x_j^{1/\alpha-1}.
  \end{align*}
  Since every index in $\{1,\dots,d\}$ appears in precisely one
  $B\in\pi$,
  \begin{align*}
    \sum_{\pi\in\Pi:|\pi|=k}\prod_{B\in\pi}\D_B
    \ell(\bm{x})&=\frac{1}{\alpha^d}\prod_{j=1}^dx_j^{1/\alpha-1}\!\!\!\!\!\sum_{\pi\in\Pi:|\pi|=k}\prod_{B\in\pi}(\alpha)_{|B|}\biggl(\,\sum_{j=1}^dx_j^{1/\alpha}\biggr)^{\alpha-|B|}\\
    &=\frac{1}{\alpha^d}\prod_{j=1}^dx_j^{1/\alpha-1}\biggl(\,\sum_{j=1}^dx_j^{1/\alpha}\biggr)^{\alpha
      k-d}\!\!\!\!\!\sum_{\pi\in\Pi:|\pi|=k}\prod_{B\in\pi}(\alpha)_{|B|}.
  \end{align*}
  Using the general form of the density as given in Proposition~\ref{prop:AXC:dens}
  and $\bm{x}=\psii(\bm{u})$ leads to the result as stated.
\end{proof}

As we can see from Proposition~\ref{prop:AXC:dens}, the general form of the
density of AXCs involves the (possibly high-order) derivatives
$\psi^{(k)}$ and $\D_B \ell$. The former are well know to be numerically
non-trivial; see, for example, \cite{hofertmaechlermcneil2012} or
\cite{hofertmaechlermcneil2013}. We therefore now address how the density of
AXCs can be computed numerically. This is typically done by
computing a proper logarithm (and then returning the exponential, but only if
required), that is, a logarithm that is numerically more robust than just
$\log c$. As we will see, two nested proper logarithms can be used to evaluate
the logarithmic density of AXCs, which is especially appealing.
\begin{proposition}[AXC logarithmic density evaluation]\label{prop:AXC:log:dens:eval}
  If the respective partial derivatives of $\ell$ exist and are continuous, the
  logarithmic density $\log c$ of a $d$-dimensional AXC $C$ is given by
  \begin{align*}
    \log c(\bm{u})=\sum_{j=1}^d \log((-\psii)'(u_j))+b_{\text{max}}^{\psi,\ell}(\bm{u})+\log\sum_{k=1}^d\exp(b_k^{\psi,\ell}(\bm{u})-b_{\text{max}}^{\psi,\ell}(\bm{u})),\quad\bm{u}\in(0,1)^d,
  \end{align*}
  where the notation is as in Proposition~\ref{prop:AXC:dens} and
  \begin{align*}
    b_k^{\psi,\ell}(\bm{u})&=\log((-1)^k\psi^{(k)})\bigl(\ell(\psii(\bm{u}))\bigr)+a_{\text{max}}^{\psi,\ell,k}(\bm{u})+\log\!\!\!\!\sum_{\pi\in\Pi:|\pi|=k}\!\!\!\!\exp(a_{\pi}^{\psi,\ell,k}(\bm{u})-a_{\text{max}}^{\psi,\ell,k}(\bm{u})),\\
    b_{\text{max}}^{\psi,\ell}(\bm{u})&=\max_{k} b_{k}^{\psi,\ell}(\bm{u})
  \end{align*}
  for
  \begin{align*}
    a_{\pi}^{\psi,\ell,k}(\bm{u})=\sum_{B\in\pi}\log((-1)^{|B|-1}\D_B \ell)(\psii(\bm{u})),\quad a_{\text{max}}^{\psi,\ell,k}(\bm{u})=\max_{\pi\in\Pi:|\pi|=k} a_{\pi}^{\psi,\ell,k}(\bm{u}).
  \end{align*}
\end{proposition}
\begin{proof}
  Let $\bm{u}\in(0,1)^d$ and note that
  \begin{align*}
    c(\bm{u})&=\biggl(\,\prod_{j=1}^d (\psii)'(u_j)\biggr)\sum_{k=1}^d\psi^{(k)}\bigl(\ell(\psii(\bm{u}))\bigr)\!\!\!\!\!\!\sum_{\pi\in\Pi:|\pi|=k}\prod_{B\in\pi}(\D_B \ell)(\psii(\bm{u}))\\
             &=\biggl(\,\prod_{j=1}^d (-\psii)'(u_j)\biggr)\sum_{k=1}^d(-1)^k\psi^{(k)}\bigl(\ell(\psii(\bm{u}))\bigr)\!\!\!\!\!\!\sum_{\pi\in\Pi:|\pi|=k}\!\!\!\!(-1)^{d-k}\prod_{B\in\pi}(\D_B \ell)(\psii(\bm{u}))\\
    &=\biggl(\,\prod_{j=1}^d (-\psii)'(u_j)\biggr)\sum_{k=1}^d(-1)^k\psi^{(k)}\bigl(\ell(\psii(\bm{u}))\bigr)\!\!\!\!\!\!\sum_{\pi\in\Pi:|\pi|=k}\prod_{B\in\pi}((-1)^{|B|-1}\D_B \ell)(\psii(\bm{u})),
  \end{align*}
  where the last equality follows from the fact that $\sum_{B\in\pi}|B|=d$ and
  $\prod_{B\in\pi}\D_B \ell$ is taken over those $\pi$ for which $|\pi|=k$, so
  $\sum_{B\in\pi}1=k$; note that, as before, $|B|$ denotes the number of
  elements of $B$.

  Since $\psi$ has derivatives with alternating signs, $(-1)^k\psi^{(k)}>0$ for
  all arguments; in particular, $(-\psii)'>0$, too. By
  \cite[Theorem~6]{ressel2013}, $\ell$ is fully $d$-max decreasing which
  implies %
  that, for all arguments of $\ell$, $\sign(\D_B \ell)=(-1)^{|B|-1}$. This
  implies that $\sign((-1)^{|B|-1}\D_B \ell)=1$ and so all terms
  $a_{\pi}^{\psi,\ell,k}$ and $b_k^{\psi,\ell}$ as defined in the claim are
  well-defined.

  Taking the logarithm, the first product in $c$ becomes
  $\sum_{j=1}^d \log((-\psii)'(u_j))$ as in the claim. By using the definitions in the claim, the logarithm of the remaining sum
  can be written as
  \begin{align}
    \log\sum_{k=1}^d\exp\Bigl(\log\Bigl((-1)^k\psi^{(k)}\bigl(\ell(\psii(\bm{u}))\bigr)\!\!\!\!\!\!\sum_{\pi\in\Pi:|\pi|=k}\prod_{B\in\pi}((-1)^{|B|-1}\D_B \ell)(\psii(\bm{u}))\Bigr)\Bigr),\label{log:dens:main:term}
  \end{align}
  where
  {\allowdisplaybreaks
  \begin{align*}
    &\phantom{={}}\log\Bigl((-1)^k\psi^{(k)}\bigl(\ell(\psii(\bm{u}))\bigr)\!\!\!\!\!\!\sum_{\pi\in\Pi:|\pi|=k}\prod_{B\in\pi}((-1)^{|B|-1}\D_B \ell)(\psii(\bm{u}))\Bigr)\\
    &=\log\bigl((-1)^k\psi^{(k)}\bigl(\ell(\psii(\bm{u}))\bigr)\bigr) + \log\!\!\!\!\sum_{\pi\in\Pi:|\pi|=k}\prod_{B\in\pi}((-1)^{|B|-1}\D_B \ell)(\psii(\bm{u}))\\
    &=\log((-1)^k\psi^{(k)})\bigl(\ell(\psii(\bm{u}))\bigr) + \log\!\!\!\!\sum_{\pi\in\Pi:|\pi|=k}\exp\Bigl(\,\sum_{B\in\pi}\log((-1)^{|B|-1}\D_B \ell)(\psii(\bm{u}))\Bigr)\\
    &=\log((-1)^k\psi^{(k)})\bigl(\ell(\psii(\bm{u}))\bigr) + \log\!\!\!\!\sum_{\pi\in\Pi:|\pi|=k}\exp(a_{\pi}^{\psi,\ell,k}(\bm{u}))\\
    &=\log((-1)^k\psi^{(k)})\bigl(\ell(\psii(\bm{u}))\bigr) + a_{\text{max}}^{\psi,\ell,k}(\bm{u}) + \log\!\!\!\!\sum_{\pi\in\Pi:|\pi|=k}\exp(a_{\pi}^{\psi,\ell,k}(\bm{u})-a_{\text{max}}^{\psi,\ell,k}(\bm{u}))\\
    &=b_k^{\psi,\ell}(\bm{u}).
  \end{align*}}
  We thus obtain that the term in \eqref{log:dens:main:term} equals
  \begin{align*}
    \log\sum_{k=1}^d\exp (b_k^{\psi,\ell}(\bm{u}))=b_{\text{max}}^{\psi,\ell}(\bm{u})+\log\sum_{k=1}^d\exp (b_k^{\psi,\ell}(\bm{u})-b_{\text{max}}^{\psi,\ell}(\bm{u})).
  \end{align*}
  Putting the terms together, the logarithmic density has the form as in the claim.
\end{proof}

A couple of remarks are in order here. First, note that due to the signs of the
involved terms, one can apply an $\exp-\log$-trick twice (nested) for computing
the logarithmic density of AXCs. The remaining logarithms of sums in
the formula of the logarithmic density are typically numerically trivial, as all
summands are bounded to lie in $[0,1]$. More importantly, the nested
$\exp-\log$-trick allows one to compute both (possibly high-order) derivatives
$\psi^{(k)}$ and $\D_B \ell$ in logarithmic scale (see $b_k^{\psi,\ell}(\bm{u})$
and $a_{\pi}^{\psi,\ell,k}(\bm{u})$, respectively); the non-logarithmic values
are never used. This is numerically an important result as the logarithmic terms
can typically be implemented efficiently themselves; for
$\log((-1)^k\psi^{(k)})$ for well known Archimedean families see, for example,
\cite{hofertmaechlermcneil2012}, \cite{hofertmaechlermcneil2013} or the \R\
package \texttt{copula} of \cite{copula}.

\section{On nested Archimax copulas}\label{sec:NAXC}
We now briefly explore the question whether, in principle, HAXCs
can also be nested copulas so \emph{nested Archimax copulas (NAXCs)}, that is,
whether there are HAXCs $C$ with analytical form
$C(\bm{u})=C_0(C_1(\bm{u}_1),\dots,C_S(\bm{u}_S))$, $\bm{u}\in[0,1]^d$. Note
that the only known nontrivial class of copulas for which such \emph{nesting}
can be done (under the sufficient nesting condition) is the class of nested
Archimedean copulas. To this end, we make the following assumption.
\begin{assumption}[Nested EVCs]\label{ass:NEVCs}
  Assume that $D_0,\dots,D_S$ are EVCs such that
  $D(\bm{u})=D_0(D_1(\bm{u}_1),\dots,D_S(\bm{u}_S))$, $\bm{u}=(\bm{u}_1,\dots,\bm{u}_S)\in[0,1]^d$, is an
  EVC.
\end{assumption}
A $D$ as in Assumption~\ref{ass:NEVCs} is referred to as \emph{nested
  extreme-value copula (NEVC)}. The only known nontrivial copula family for
which Assumption~\ref{ass:NEVCs} is known to hold is the nested Gumbel family
(under the sufficient nesting condition). It thus remains an open question
whether there are other families of EVCs or a general construction of NEVCs
besides the Gumbel.

\subsection{Based on nested extreme-value copulas or nested stable tail
  dependence functions}
Our first result shows that Assumption~\ref{ass:NEVCs} is equivalent to the
existence of a \emph{nested stable tail dependence function}.
\begin{lemma}[Nesting correspondence]\label{lem:nesting:stable:tail:dep}
  An EVC $D$ is a NEVC if and only if the stable tail dependence function $\ell$
  of $D$ is \emph{nested}, that is,
  \begin{align}
    \ell(\bm{x})=\ell_0(\ell_1(\bm{x}_1),\dots,\ell_S(\bm{x}_S)),\quad\bm{x}\in[0,\infty)^d.\label{eq:ntdf}
  \end{align}
\end{lemma}
\begin{proof}
  \begin{align*}
    D(\bm{u})&=D_0(D_1(\bm{u}_1),\dots,D_S(\bm{u}_S))=\exp(-\ell_0(-\log D_1(\bm{u}_1),\dots,-\log D_S(\bm{u}_S)))\\
    &=\exp\Bigl(-\ell_0\bigl(-\log\bigl(\exp(-\ell_1(-\log u_{11},\dots,-\log
      u_{1d_1}))\bigr),\dots,\\
    &\phantom{=\exp\Bigl(-\ell_0\bigl(}-\log\bigl(\exp(-\ell_S(-\log u_{S1},\dots,-\log u_{Sd_S}))\bigr)\bigr)\Bigr)\\
    &=\exp\bigl(-\ell_0(\ell_1(-\log u_{11},\dots,-\log u_{1d_1}),\dots,\ell_S(-\log u_{S1},\dots,-\log u_{Sd_S}))\bigr)\\
    &=\exp(-\ell(-\log u_{11},\dots,-\log u_{Sd_S})),\quad\bm{u}\in[0,1]^d,
  \end{align*}
  if and only if $\ell(\bm{x})=\ell_0(\ell_1(\bm{x}_1),\dots,\ell_S(\bm{x}_S)),\quad\bm{x}\in[0,\infty)^d$.
\end{proof}

The following proposition is essentially a nested version of one of the two HAXC
extensions suggested in Section~\ref{sec:two:ways:HAXCs} which, based on
Assumption~\ref{ass:NEVCs} leads to \emph{nested AXCs (NAXCs)} based on NEVCs
or, equivalently, nested stable tail dependence functions; see
Lemma~\ref{lem:nesting:stable:tail:dep}.
\begin{proposition}[NAXCs based on NEVCs or nested stable tail dependence functions]\label{thm:NAXCs:via:NEVCs}
  Let $D_s$, $s\in\{0,\dots,S\}$, be as in Assumption~\ref{ass:NEVCs} with
  respective stable tail dependence functions $\ell_s$, $s\in\{0,\dots,S\}$. Let $V\sim
  F=\LSi[\psi]$ and
  $\bm{Y}=(\bm{Y}_1,\dots,\bm{Y}_S)=(Y_{11},\dots,Y_{1d_1},\dots,$ $Y_{S1},\dots,Y_{Sd_S})\sim
  D$ be independent, where $D$ is an EVC as in Assumption~\ref{ass:NEVCs}. Then the copula $C$ of
  \begin{align*}
    \bm{U}&=\Bigl(\psi\Bigl(\frac{-\log\bm{Y}_1}{V}\Bigr),\dots,\psi\Bigl(\frac{-\log\bm{Y}_S}{V}\Bigr)\Bigr)\\
          &=\Bigl(\psi\Bigl(\frac{-\log Y_{11}}{V}\Bigr),\dots,\psi\Bigl(\frac{-\log Y_{1d_1}}{V}\Bigr),\dots,\psi\Bigl(\frac{-\log Y_{S1}}{V}\Bigr),\dots,\psi\Bigl(\frac{-\log Y_{Sd_S}}{V}\Bigr)\Bigr)
  \end{align*}
  is given, for all $\bm{u}\in[0,1]^d$, by
  \begin{align*}
    C(\bm{u})&=\psi\bigl(\ell_0\bigl(\ell_1(\psii(\bm{u}_1)),\dots,\ell_S(\psii(\bm{u}_S))\bigr)\bigr)\\
    &=\psi\bigl(\ell_0\bigl(\ell_1(\psii(u_{11}),\dots,\psii(u_{1d_1})),\dots,\ell_S(\psii(u_{S1}),\dots,\psii(u_{Sd_S}))\bigr)\bigr);
  \end{align*}
  that is, $C$ is an AXC with nested stable tail dependence function as given in \eqref{eq:ntdf}.
\end{proposition}
\begin{proof}
  \begin{align*}
    \P(\bm{U}\le\bm{u})&=\P(\bm{Y}_1\le e^{-V\psii(\bm{u}_1)},\dots,\bm{Y}_S\le e^{-V\psii(\bm{u}_S)})\\
    &=\E(\P(\bm{Y}_1\le e^{-V\psii(\bm{u}_1)},\dots,\bm{Y}_S\le e^{-V\psii(\bm{u}_S)}\,|\,V))\\
    &=\E(D(e^{-V\psii(\bm{u}_1)},\dots,e^{-V\psii(\bm{u}_S)}))=\E(D^V(e^{-\psii(\bm{u}_1)},\dots,e^{-\psii(\bm{u}_S)}))\\
    &=\E(\exp\bigl(-V\ell(\psii(\bm{u}_1),\dots,\psii(\bm{u}_S))\bigr))=\psi\bigl(\ell(\psii(\bm{u}_1),\dots,\psii(\bm{u}_S))\bigr)
  \end{align*}
  The claim immediately follows from Lemma~\ref{lem:nesting:stable:tail:dep} by noting
  that $D$ is nested as of Assumption~\ref{ass:NEVCs}.
\end{proof}

\begin{corollary}[Pairwise marginal copulas]
  Under the setup of Proposition~\ref{thm:NAXCs:via:NEVCs} the bivariate marginal
  copulas of $C$ satisfy
  \begin{align*}
    C(u_{si},u_{tj})=\begin{cases}
      \psi(\ell_s(\psii(u_{si}),\psii(u_{sj}))),&\text{if}\,\ t=s,\\
      \psi(\ell_0(\psii(u_{si}),\psii(u_{tj}))),&\text{otherwise}.
    \end{cases}
  \end{align*}
  Therefore, the bivariate marginal copulas of $C$ are (possibly different) AXCs.
\end{corollary}
\begin{proof}
  For a stable tail dependence function $\ell$, one has that $\ell(\bm{x})=x_j$ if
  all components except the $j$th of $\bm{x}$ are 0. As such, for any
  $s\in\{1,\dots,S\}$,
  \begin{align*}
    \ell_s(\psii(u_{s1}),\dots,\psii(u_{sd_s}))=\begin{cases}
      0,&\text{if}\,\ u_{sj}=1\ \forall\,j\in\{1,\dots,d_s\},\\
      \psii(u_{sk}),&\text{if}\,\ u_{sj}=1\ \forall\,j\in\{1,\dots,d_s\}\backslash\{k\},\\
      \ell_s(\psii(u_{sk}),\psii(u_{sl})),&\text{if}\,\ u_{sj}=1\ \forall\,j\in\{1,\dots,d_s\}\backslash\{k,l\},
    \end{cases}
  \end{align*}
  from which the result follows.
\end{proof}

\subsection{Additionally nesting frailties}
As in the second method for introducing hierarchies on AXCs presented in
Section~\ref{sec:two:ways:HAXCs}, we could, additionally, impose a hierarchical
structure on the underlying (multiple) frailties. We focus on the two-level case
with $S$ different frailties. Assume, as before, the sufficient nesting
condition to hold, that is, $\psi_s\in\Psi$, $s\in\{0,\dots,S\}$, are
Archimedean generators and, for all $s\in\{0,\dots,S\}$, the derivative of
$\psiis{0}\circ\psi_s$ is completely monotone.

\begin{proposition}[NAXCs based on nested frailties]
  Let $D_s$, $s\in\{0,\dots,S\}$, be as in Assumption~\ref{ass:NEVCs} with
  respective stable tail dependence functions $\ell_s$,
  $s\in\{0,\dots,S\}$. Furthermore, let $\psi_s\in\Psi_\infty$,
  $s\in\{0,\dots,S\}$, and assume that the sufficient nesting condition holds. Assume
  $V_0\sim F_0=\LSi[\psi_0]$ and
  $V_{0s}\,|\,V_0\sim F_{0s}=\LSi[\psi_{0s}(\cdot\,;V_0)]$,
  $s\in\{1,\dots,S\}$. Moreover, let
  $\bm{Y}=(\bm{Y}_1,\dots,\bm{Y}_S)\sim D$ be independent of
  $V_0,V_1,\dots,V_S$ and assume that
  \begin{align}
    &\phantom{{}={}}\E\bigl(\E\bigl(D_0(D_1(e^{-V_{01}\psiis{1}(\bm{u}_{1})}),\dots,D_S(e^{-V_{0S}\psiis{S}(\bm{u}_{S})}))\,\big|\,V_0\bigr)\bigr)\notag\\
    &=\E\bigl(D_0\bigl(\E(D_1(e^{-V_{01}\psiis{1}(\bm{u}_{1})})\,|\,V_0),\dots,\E(D_S(e^{-V_{0S}\psiis{S}(\bm{u}_{S})})\,|\,V_0)\bigr)\bigr).\label{crucial:cond}
  \end{align}
  Then the copula $C$ of
  \begin{align*}
    \bm{U}=\Bigl(\psi_1\Bigl(\frac{-\log\bm{Y}_1}{V_{01}}\Bigr),\dots,\psi_S\Bigl(\frac{-\log\bm{Y}_S}{V_{0S}}\Bigr)\Bigr)
  \end{align*}
  is given by
  \begin{align*}
    C(\bm{u})=C_0(C_1(\bm{u}_1),\dots,C_S(\bm{u}_S)),\quad\bm{u}\in[0,1]^d,
  \end{align*}
  where, for all $s\in\{0,\dots,S\}$, $C_s$ is Archimax with generator $\psi_s$ and stable tail dependence function
  $\ell_s$.
\end{proposition}
\begin{proof}
  \begin{align*}
    \P(\bm{U}\le\bm{u})&=\P(\bm{Y}_{1}\le e^{-V_{01}\psiis{1}(\bm{u}_{1})},\dots,\bm{Y}_{S}\le e^{-V_{0S}\psiis{S}(\bm{u}_{S})})\\
    &=\E\bigl(\E(\P(\bm{Y}_{1}\le e^{-V_{01}\psiis{1}(\bm{u}_{1})},\dots,\bm{Y}_{S}\le e^{-V_{0S}\psiis{S}(\bm{u}_{S})}\,|\,V_{01},\dots,V_{0S})\,|\,V_0)\bigr)\\
    &=\E\bigl(\E(D(e^{-V_{01}\psiis{1}(\bm{u}_{1})},\dots,e^{-V_{0S}\psiis{S}(\bm{u}_{S})})\,|\,V_0)\bigr)\\
    &\omuc{}{=}{\eqref{crucial:cond}}\E\bigl(D_0\bigl(\E(D_1(e^{-V_{01}\psiis{1}(\bm{u}_{1})})\,|\,V_0),\dots,\E(D_S(e^{-V_{0S}\psiis{S}(\bm{u}_{S})})\,|\,V_0)\bigr)\bigr).
  \end{align*}
  Each component $\E(D_s(e^{-V_{0s}\psiis{s}(\bm{u}_{s})})\,|\,V_0)$, $s\in\{1,\dots,S\}$, satisfies
  \begin{align*}
    \E(D_s(e^{-V_{0s}\psiis{s}(\bm{u}_{s})})\,|\,V_0)&=\E(D_s^{V_{0s}}(e^{-\psiis{s}(\bm{u}_{s})})\,|\,V_0)=\E(e^{-V_{0s}\ell_s(\psiis{s}(\bm{u}_{s}))}\,|\,V_0)\\
    &=\psi_{0s}(\ell_s(\psiis{s}(\bm{u}_s));V_0),
  \end{align*}
  thus
  \begin{align*}
    \P(\bm{U}\le\bm{u})&=\E\bigl(D_0\bigl(\psi_{01}(\ell_1(\psiis{1}(\bm{u}_1));V_0),\dots,\psi_{0S}(\ell_S(\psiis{S}(\bm{u}_S));V_0)\bigr)\bigr)\\
    &=\E\bigl(D_0\bigl(e^{-V_0\psiis{0}(C_1(\bm{u}_1))},\dots,e^{-V_0\psiis{0}(C_S(\bm{u}_S))}\bigr)\bigr)\\
    &=\E\bigl(D_0^{V_0}\bigl(e^{-\psiis{0}(C_1(\bm{u}_1))},\dots,e^{-\psiis{0}(C_S(\bm{u}_S))}\bigr)\bigr)\\
    &=\E\Bigl(e^{-V_0\bigl(\ell_0\bigl(\psiis{0}(C_1(\bm{u}_1)),\dots,\psiis{0}(C_S(\bm{u}_S))\bigr)\bigr)}\Bigr)\\
    &=\psi_0\bigl(\ell_0\bigl(\psiis{0}(C_1(\bm{u}_1)),\dots,\psiis{0}(C_S(\bm{u}_S))\bigr)\bigr)=C_0(C_1(\bm{u}_1),\dots,C_S(\bm{u}_S)).\qedhere
  \end{align*}
\end{proof}

The following corollary provides a condition under which
Assumption~\eqref{crucial:cond} holds. Note that this particular model can
already be found in \cite{mcfadden1978}.
\begin{corollary}[AC composed with AXCs]
  If $D(\bm{u})=\prod_{s=1}^S D_s(\bm{u}_s)$, \eqref{crucial:cond} holds
  and $C(\bm{u})=C_0(C_1(\bm{u}_1),\dots,C_S(\bm{u}_S))$, where $C_0$ is
  Archimedean and $C_1,\dots,C_S$ are Archimax. In particular, if $D$ is the independence
  copula, \eqref{crucial:cond} holds and $C$ is a NAC.
\end{corollary}
\begin{proof}
  If $D(\bm{u})=\prod_{s=1}^S D_s(\bm{u}_s)$, then, conditional on $V_0$, the
  sector components are independent and we obtain
  \begin{align*}
    &\phantom{{}={}}\E\bigl(D_0(D_1(e^{-V_{01}\psiis{1}(\bm{u}_{1})}),\dots,D_S(e^{-V_{0S}\psiis{S}(\bm{u}_{S})}))\,\big|\,V_0\bigr)\\
    &=\E\Bigl(\,\prod_{s=1}^SD_s(e^{-V_{0s}\psiis{s}(\bm{u}_{s})})\,\Big|\,V_0\Bigr)=\prod_{s=1}^S\E(D_s(e^{-V_{0s}\psiis{s}(\bm{u}_{s})})\,|\,V_0)\\
    &=D_0\bigl(\E(D_1(e^{-V_{01}\psiis{1}(\bm{u}_{1})})\,|\,V_0),\dots,\E(D_S(e^{-V_{0S}\psiis{S}(\bm{u}_{S})})\,|\,V_0)\bigr).
  \end{align*}
  and thus \eqref{crucial:cond} follows by taking the expectation. The rest
  follows immediately by noting that an EVC is the independence
  copula if and only if its stable tail dependence function is the sum of its components,
  so the Archimax (sector) copulas
  $C_s(\bm{u}_s)=\psi_s\bigl(\ell_s(\psiis{s}(u_{s1}),\dots,\psiis{s}(u_{sd_s}))\bigr)$ are
  Archimedean generated by $\psi_s$, $s\in\{1,\dots,S\}$.
\end{proof}

\printbibliography[heading=bibintoc]
\end{document}

%
%
%
%
